\documentclass[11pt]{amsart}
\usepackage[utf8x]{inputenc}
\usepackage[english]{babel}
\usepackage[T1]{fontenc}
 
\usepackage{graphicx}
\usepackage{caption}
\usepackage{subcaption}
\usepackage{amssymb,amsmath}
\usepackage[hidelinks]{hyperref}
\usepackage{indentfirst}
\usepackage{enumerate,amsmath,amssymb, mathrsfs,mathtools}
\usepackage{appendix}
\usepackage{latexsym}
\usepackage{url}
\usepackage{color}
\usepackage{accents}
\usepackage{setspace}
\usepackage{varwidth}

\usepackage{pdfpages}

\usepackage[margin=2.5cm]{geometry}
\allowdisplaybreaks

\newcommand{\hcirc}{\accentset{\circ}{h}}

\mathcode`l="8000
\begingroup
\makeatletter
\lccode`\~=`\l
\DeclareMathSymbol{\lsb@l}{\mathalpha}{letters}{`l}
\lowercase{\gdef~{\ifnum\the\mathgroup=\m@ne \ell \else \lsb@l \fi}}%
\endgroup

\def\XXint#1#2#3{{\setbox0=\hbox{$#1{#2#3}{\int}$ }
		\vcenter{\hbox{$#2#3$ }}\kern-.6\wd0}}

\newtheorem{prop}{Proposition}
\newtheorem{thm}[prop]{Theorem}
\newtheorem{lem}[prop]{Lemma}
\newtheorem{coro}[prop]{Corollary}

\newtheorem{rema}[prop]{Remark}

\title[The Willmore center of mass of initial data sets
]{	The Willmore center of mass of initial data sets
}
\author{Michael Eichmair }
\address{ University of Vienna,
	Faculty of Mathematics,
	Oskar-Morgenstern-Platz 1,
	1090 Vienna,
	Austria \newline \indent
ORCiD: \href{https://orcid.org/0000-0001-7993-9536}{0000-0001-7993-9536}}
\email{michael.eichmair@univie.ac.at}
\author{Thomas Koerber}
\address{ University of Vienna,
	Faculty of Mathematics,
	Oskar-Morgenstern-Platz 1,
	1090 Vienna,
	Austria
\newline \indent
ORCiD: \href{https://orcid.org/0000-0003-1676-0824}{0000-0003-1676-0824}}
\email{thomas.koerber@univie.ac.at}

\begin{document}
	
\date{\today}
\onehalfspacing

\begin{abstract}

		We refine the Lyapunov-Schmidt analysis from our recent paper \cite{acws} to study the geometric center of mass of the asymptotic foliation by area-constrained Willmore surfaces of initial data for the Einstein field equations. 
	If the scalar curvature of the initial data vanishes at infinity, we show that this geometric center of mass agrees with the Hamiltonian center of mass.
	By contrast, we show that the positioning of large area-constrained Willmore surfaces is sensitive to the distribution of the energy density.
	 In particular, the geometric center of mass may differ from the Hamiltonian center of mass if the scalar curvature does not satisfy additional asymptotic symmetry assumptions.
\end{abstract}
\maketitle
\section{Introduction}	
Let $(M,g)$ be an asymptotically flat Riemannian $3$-manifold. Such Riemannian manifolds are used to model initial data of isolated gravitational systems for the Einstein field equations. The scalar curvature  of $(M,g)$ provides a lower bound for the local energy density of the initial data set. The  geometry of $(M,g)$ encodes global invariants of the evolving gravitating system. 
\\ \indent Recall that the  mass $m\in\mathbb{R}$ of such a manifold $(M,g)$,  proposed by R.~Arnowitt, S.~Deser, and C.~W.~Misner in \cite{ArnowittDeserMisner}, can be computed as a limit of flux integrals. More precisely,
\begin{align} \label{ADM mass}
m=\lim_{\lambda\to\infty}\sum_{i,\,j=1}^3\,\frac{1}{16\,\pi}\,\lambda^{-1}\int_{S_\lambda(0)}x^j\,\left[ (\partial_ig)(e_i,e_j)-(\partial_jg)(e_i,e_i)\right]\,\mathrm{d}\bar\mu
\end{align}
where the integrals are computed in an asymptotically flat chart of  $(M,g)$. R.~Bartnik \cite{Bartnik} has shown that the limit in the definition of \eqref{ADM mass} exists and does not depend on the particular choice of chart. If $(M,g)$ has non-negative scalar curvature and  is  not isometric to $\mathbb{R}^3$, R.~Schoen and S.-T.~Yau \cite{SchoenYau} and E.~Witten \cite{Witten} have shown that $m>0$. The Hamiltonian center of mass associated with $(M,g)$, proposed by T.~Regge and C.~Teitelboim \cite{ReggeTeitelboim} and by R.~Beig and N.~\'{O} Murchadha \cite{BeigOMurchada},
is then given by $C=(C^1,C^2,C^3)$ where
\begin{equation}
\begin{aligned}
C^\ell =\,&\lim_{\lambda\to\infty}\frac{1}{16\,\pi\,m}\,\lambda^{-1}\,\int_{S_\lambda(0)}\bigg(\,\sum_{i,\,j=1}^3x^\ell\,x^j\,\big[(\partial_ig)(e_i,e_j)-(\partial_jg)(e_i,e_i)\big]
\\&\qquad \qquad\qquad\qquad\qquad\, \quad-\sum_{i=1}^3\big[x^i\,g(e_i,e_\ell)-x^\ell\,g(e_i,e_i)\big]\bigg)\,\mathrm{d}\bar\mu
\end{aligned}
\label{center of mass}
\end{equation}
provided the limit  exists for each $\ell=1,\,2,\,3$. These limits are known to exist if if $g$ satisfies certain asymptotic symmetry conditions; see Theorem \ref{com existence criterion} below. By contrast, as  observed in \cite{ReggeTeitelboim}, the limit in \eqref{center of mass} may not exist if $g$ does not satisfy such additional assumptions. Explicit examples of asymptotically flat initial data with divergent center of mass have been constructed by R.~Beig and N.~\'{O} Murchadha \cite{BeigOMurchada}, by L.-H.~Huang \cite{Huang3}, and by C.~Cederbaum and C.~Nerz \cite{cederbaumnerz}. 
\\
\indent Let $\Sigma\subset M$ be a closed, two-sided surface with designated outward normal $\nu$ and corresponding mean curvature $H$. The Hawking mass of $\Sigma$ is the quantity
\begin{align} \label{hawking mass}
m_H(\Sigma)=\sqrt{\frac{|\Sigma|}{16\,\pi}}\bigg(1-\frac{1}{16\,\pi}\int_{\Sigma} H^2\,\mathrm{d}\mu\bigg).
\end{align} To qualify as a quasi-local mass in the sense of \cite[p.~235]{BartnikQuasi}, one would expect that the Hawking mass both detects the local energy distribution and recovers global physical quantities such as the mass \eqref{ADM mass} and the center of mass \eqref{center of mass} of the initial data set as asymptotic limits; see also \cite[p.~636]{Penrose}. In \cite{HuiskenIlmanen}, G.~Huisken and T.~Ilmanen have proved the Riemannian Penrose inequality by comparing  the Hawking mass of an outermost minimal surface to that of a large coordinate sphere in the end of $(M,g)$ using inverse mean curvature flow. By contrast, the Hawking mass of a closed surface $\Sigma\subset\mathbb{R}^3$ is negative unless $\Sigma$ is a round sphere. As a measure of the gravitational field, the quantity $m_H(\Sigma)$ is therefore not appropriate unless $\Sigma$ is in some way special.\\ \indent  As discussed in e.g.~\cite{acws}, there are two classes of surfaces that are particularly well-adapted to the Hawking mass: 
\begin{enumerate}
	\item[$\circ$] stable constant mean curvature spheres
	\item[$\circ$] area-constrained Willmore spheres
\end{enumerate}
 In \cite{ChristodoulouYau}, D.~Christodoulou and S.-T.~Yau have observed that stable constant mean curvature spheres have non-negative Hawking mass if $(M,g)$ has non-negative scalar curvature. Meanwhile, area-constrained Willmore surfaces are by definition critical points of the Hawking mass with respect to an area constraint and thus potential maximizers of the Hawking mass among domains with a prescribed amount of perimeter. These surfaces satisfy the constrained Willmore equation
\begin{align}
\Delta H+(|\hcirc|^2+\operatorname{Ric}(\nu,\nu)+\kappa)\,H=0.
\label{constrained Willmore quantity}
\end{align}
Here, $\Delta$ is the non-positive Laplace-Beltrami operator with respect to the induced metric on $\Sigma$, $\hcirc$ the traceless part of the second fundamental form $h$,  $\operatorname{Ric}$ the Ricci curvature of $(M,g)$, and $\kappa\in\mathbb{R}$ a Lagrange multiplier. Note that area-constrained Willmore surfaces are also area-constrained critical points of the Willmore energy
\begin{align*} 
\frac14\int_{\Sigma} H^2\,\mathrm{d}\mu.
\end{align*}
 T.~Lamm, J.~Metzger, and F.~Schulze have studied foliations of asymptotically flat Riemannian 3-manifolds by  area-constrained Willmore surfaces and investigated the monotonicity properties of the Hawking mass along this foliation; see \cite[Theorem 1, Theorem 2, and Theorem 4]{LammMetzgerSchulze} and also the subsequent work \cite{koerber} of the second-named author. Results analogous to those in \cite{LammMetzgerSchulze} but in a space-time setting have been obtained by A.~Friedrich in \cite{friedrich}.  \\
\indent
There have been many recent developments on large stable constant mean curvature spheres in asymptotically flat manifolds. In particular, it is known that the end of every asymptotically flat $3$-manifold $(M,g)$ with non-negative scalar curvature is foliated by large isoperimetric surfaces. This foliation detects the Hamiltonian center of mass \eqref{center of mass} of $(M,g)$ in a natural way. We provide a brief survey of these results in Appendix \ref{cmc appendix}.
\\ \indent
By comparison, much less is known about area-constrained Willmore surfaces. To describe recent developments, given an integer $k\geq 2$, we say that $(M,g)$ is $C^k$-asymptotic to Schwarzschild with mass $m>0$ if there is a non-empty compact subset of $M$ whose complement is diffeomorphic to $\{x\in\mathbb{R}^3:|x|>1\}$ and such that, in this so-called asymptotically flat chart of the end of $(M,g)$, there holds,  for every multi-index $J$ with $|J|\leq k$ and as $x\to\infty$,
\begin{align} 
g=\bigg(1+\frac{m}{2\,|x|}\bigg)^4\bar g+\sigma
\qquad\text{with}\qquad  \partial_J\sigma =O(|x|^{-2-|J|}).
 \label{schwarzschild sigma decay} \end{align}  
Here,  $\bar g$ is the Euclidean metric on $\mathbb{R}^3$.  Note that $(M,g)$ is modeled upon the initial data of a Schwarzschild black hole. 
Given such a manifold $(M,g)$, we fix an asymptotically flat chart and use $B_r$, where $r>1$, to denote the open, bounded domain in $(M,g)$ whose boundary corresponds to $S_r(0)$ with respect to this chart.  \\ \indent 
 In our recent paper \cite{acws}, we have  established the following existence and uniqueness result. For its statement, recall that
the area radius $\lambda(\Sigma)>0$ of a closed surface $\Sigma\subset M$ is defined by $$4\,\pi\, \lambda^2(\Sigma)=|\Sigma|$$ while the inner radius $\rho(\Sigma)$ of such a surface is defined by
$$
\rho(\Sigma)=\sup\{r>1: B_{r}\cap\Sigma=\emptyset\}.$$
\begin{thm}[{\cite[Theorems 5 and 8]{acws}}] 
Suppose that $(M,g)$ is $C^4$-asymptotic to Schwarzschild with mass $m>0$ and that its scalar curvature $R$ satisfies, as $x\to\infty$, \label{acw existence}
\begin{align*}
\sum_{i=1}^3x^i\,\partial_i(|x|^2\,R)\leq o(|x|^{-2})\qquad\text{and}\qquad  R(x)-R(-x)=o(|x|^{-4}). 
\end{align*}
Then there exist numbers $\kappa_0>0$ and $\epsilon_0>0$ and a family of stable area-constrained Willmore spheres \begin{align} \label{acw foliation}
\{\Sigma(\kappa):\kappa\in(0,\kappa_0)\}
\end{align}
that foliate the complement of a compact subset of $M$ and such that each sphere $\Sigma(\kappa)$ satisfies \eqref{constrained Willmore quantity} with  parameter $\kappa$. Moreover,
given $\delta>0$, there exists a number $\lambda_0>1$ such that every area-constrained Willmore sphere $\Sigma\subset M$ with
$$
\delta\,\lambda(\Sigma)<\rho(\Sigma),\qquad \delta\,\rho(\Sigma)<\lambda(\Sigma),\qquad |\Sigma|>4\,\pi\,\lambda_0^2,\qquad\text{and}\qquad \int_{\Sigma}|\hcirc|^2\,\mathrm{d}\mu<\epsilon_0
$$ satisfies $\Sigma=\Sigma(\kappa)$ for some $\kappa\in(0,\kappa_0)$.
\end{thm}
 The canonical foliation by area-constrained Willmore surfaces given in Theorem \ref{acw existence}  gives rise to a new notion of geometric center of mass,
\[C_{ACW}=(C_{ACW}^1,C_{ACW}^2,C_{ACW}^3),\] 
 where
\begin{align} \label{acwcom def}
C_{ACW}^\ell=\lim_{\kappa\to0} |\Sigma(\kappa)|^{-1}\int_{\Sigma(\kappa)} x^\ell\,\mathrm{d}\mu
\end{align}
provided this limit exists for each $\ell=1,\,2,\,3$.
\subsection*{Outline of the results}
Our first main result in this paper shows that the geometric center of mass of the foliation \eqref{acw foliation} exists and agrees with the Hamiltonian center of mass \eqref{center of mass} of $(M,g)$  if the scalar curvature is sufficiently symmetric with respect to the Hamiltonian center of mass.
\begin{thm}
	Let $(M,g)$ be $C^4$-asymptotic to Schwarzschild with mass $m>0$ and Hamiltonian center of mass $C=(C^1,\,C^2,\,C^3)$ and suppose that the scalar curvature satisfies, as $x\to\infty$,  \label{acwt1}
	\begin{align}  \sum_{i=1}^3\tilde x^i\,\partial_i(|\tilde x|^2\, R(\tilde x))&\leq o(|x|^{-3}),
	\label{improved center} \\ 	R(\tilde x)-R(-\tilde x)&=o(|x|^{-5}), 
	\label{improved center 2} 
	\end{align}
	 where $\tilde x=x-C$. 	 Then $C_{ACW}$ exists and $C=C_{ACW}.$ 
\end{thm} 
In particular, if $R=0$ outside a compact set, then $C_{ACW}$ exists and equals the Hamiltonian center of mass $C$.
\begin{rema}
	According to Theorem \ref{com existence criterion} and Remark \ref{com existence}, if
	\begin{align*}  R(x)-R(-x)=O(|x|^{-5}), 
	\end{align*}
	then $C$ exists.
\end{rema}
\begin{rema}
	The assumptions \eqref{improved center} and \eqref{improved center 2} of Theorem \ref{acwt1} hold if, for instance,
	$$
	R=o(|x|^{-4})\qquad \text{and} \qquad R(x)-R(-x)=o(|x|^{-5});
	$$
	see the argument leading to \eqref{dweakercenter}.
\end{rema}
The following result shows that the assumptions \eqref{improved center} and \eqref{improved center 2} in Theorem \ref{acwt1} cannot be relaxed in any substantial way.
\begin{thm}
	There exists a Riemannian $3$-manifold $(M,g)$ that is $C^k$-asymptotic to Schwarzschild with mass $m=2$ for every $k\geq 2$ and satisfies, for every multi-index $J$ and  as $x\to\infty$,
	$$
	\partial_J\sigma=O(|x|^{-3-|J|})
	$$
	  such that the Hamiltonian center of mass  exists while \label{com counterexample 2}
	the limit in \eqref{acwcom def} does not exist.
\end{thm}

\indent Theorem \ref{com counterexample 2} and its proof show that, in general, the positioning of the foliation \eqref{acw foliation} is not governed by the Hamiltonian center of mass of $(M,g)$ but instead fine-tuned to the asymptotic distribution of  scalar curvature; see Remark \ref{positioning remark}. By contrast, the positioning of large stable constant mean curvature spheres is not sensitive to the  distribution of scalar curvature; see Remark \ref{cmc positioning}. This suggests that large area-constrained Willmore spheres are better suited to detect the local energy distribution of an initial data set than large stable constant mean curvature spheres.
\\ \indent  In the second part of this paper, we lay the foundation to investigate the interplay between the positioning of area-constrained Willmore surfaces and the asymptotic distribution of the scalar curvature  more thoroughly by extending Theorem \ref{acw existence} to manifolds $(M,g)$ that are asymptotic to Schwarzschild but whose scalar curvature does not exhibit any asymptotic symmetries beyond those implied by \eqref{schwarzschild sigma decay}.

\begin{thm}
	Let $(M,g)$ be $C^4$-asymptotic to Schwarzschild with mass $m>0$ and scalar curvature $R$ satisfying, as $x\to\infty$,
	\begin{align}
	R\geq-o(|x|^{-4}). \label{r lower bound}
	\end{align}
  There exist a number  $\kappa_0>0$ and a family $\{\Sigma(\kappa):\kappa\in(0,\kappa_0)\}$ of   area-constrained Willmore spheres $\Sigma(\kappa)$ such that $\Sigma(\kappa)$ satisfies \eqref{constrained Willmore quantity} with parameter $\kappa$ and 
	$$
	\lim_{\kappa\to0} \rho(\Sigma(\kappa))=\infty, \qquad 
	\limsup_{\kappa\to0} \rho(\Sigma(\kappa))^{-1}\,\lambda(\Sigma(\kappa))<\infty, \qquad \text{and}\qquad \lim_{\kappa\to0} \int_{\Sigma(\kappa)} |\hcirc|^2\,\mathrm{d}\mu=0.
	$$
	Moreover, if the scalar curvature satisfies, as $x\to\infty$, \label{general existence thm}
	\begin{align}
\sum_{i=1}^3	x^i\,\partial_i(|x|^2\,R)&\leq o(|x|^{-2}), \label{growth with error 2}
	\end{align}
	then there exists a number $\epsilon_0>0$ with the following property. Given $\delta>0$, there exists a number $\lambda_0>1$ such that every area-constrained Willmore sphere $\Sigma\subset M$ with
	\begin{align} \label{required properties} 
\delta\,	\lambda(\Sigma)<\rho(\Sigma),\qquad\delta\, \rho(\Sigma)<\lambda(\Sigma),\qquad |\Sigma|>4\,\pi\,\lambda_0^2,\qquad \text{and}\qquad\int_{\Sigma}|\hcirc|^2\,\mathrm{d}\mu<\epsilon_0
	\end{align}  satisfies $\Sigma=\Sigma(\kappa)$ for some $\kappa\in(0,\kappa_0)$.
\end{thm}

\begin{rema}
	The assumption $\delta\,\rho(\Sigma)<\lambda(\Sigma)$ in \eqref{required properties} can be dropped if one replaces \eqref{growth with error 2} by the stronger condition \label{no far outlying}
	$$
	\sum_{i=1}^3x^i\,\partial_i(|x|^2\,R)\leq 0;
	$$
	see \cite[Theorem 11]{acws}. 
\end{rema}
\begin{rema}
	Note that \eqref{r lower bound} follows from \eqref{growth with error 2}.
\end{rema} 
\begin{rema}  Comparing Theorem \ref{acw existence} and Theorem \ref{general existence thm}, it is tempting to conjecture that the asymptotic family $\{\Sigma(\kappa):\kappa\in(0,\kappa_0)\}$ from Theorem \ref{general existence thm} forms a foliation. A closer analysis shows that the foliation property of this family depends on the asymptotic behavior of the scalar curvature in a  delicate way. We plan to investigate this dependence in a future paper.
\end{rema}
\begin{figure}
	\centering
	\fbox{\begin{varwidth}{0.31\textwidth}
	\begin{subfigure}{\textwidth}
	\includegraphics[width=1\linewidth]{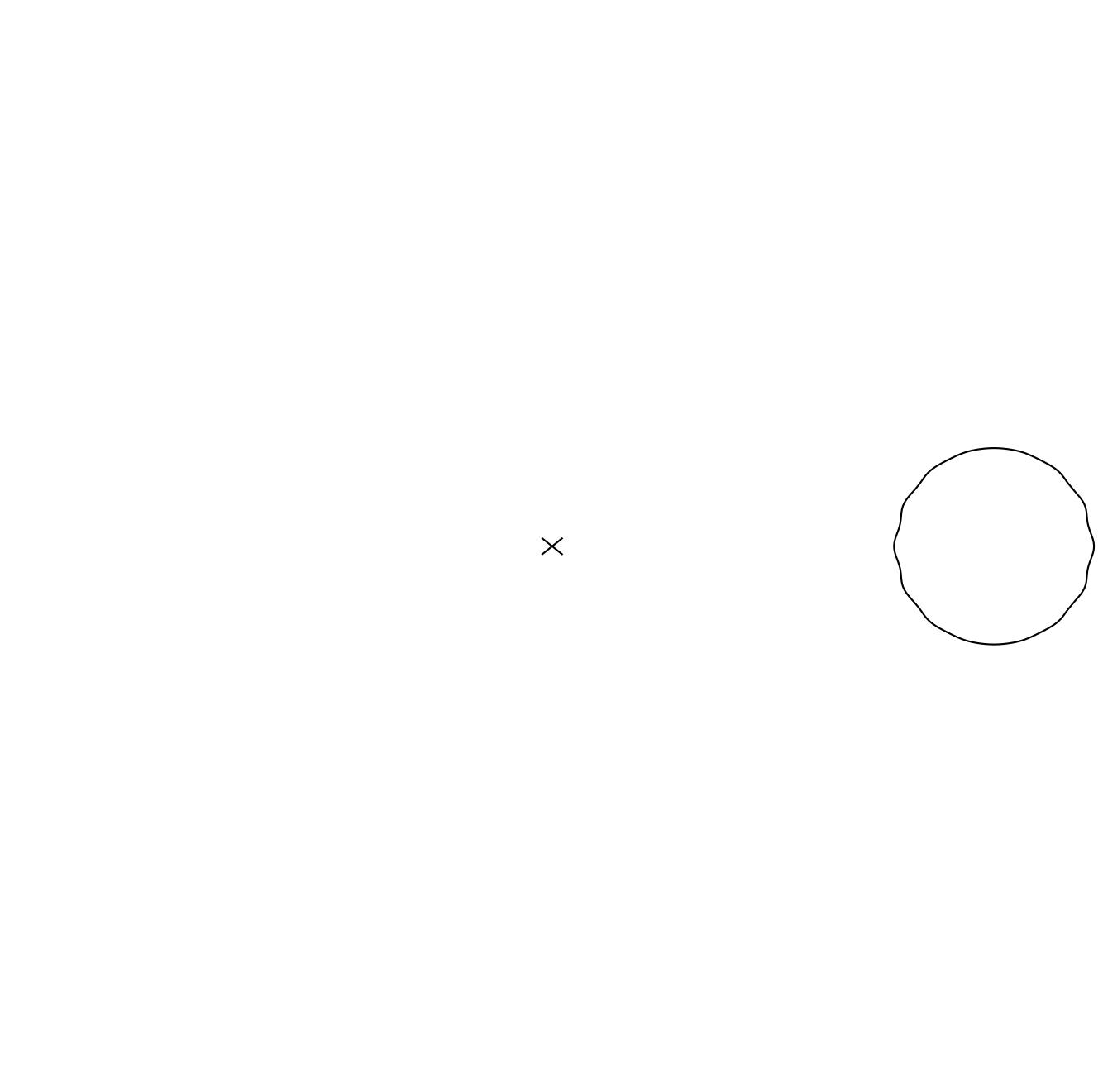}
\end{subfigure}
\end{varwidth}}
\hspace{-7.7pt}
	\fbox{\begin{varwidth}{0.31\textwidth}
		\begin{subfigure}{\textwidth}
			
			\includegraphics[width=1\linewidth]{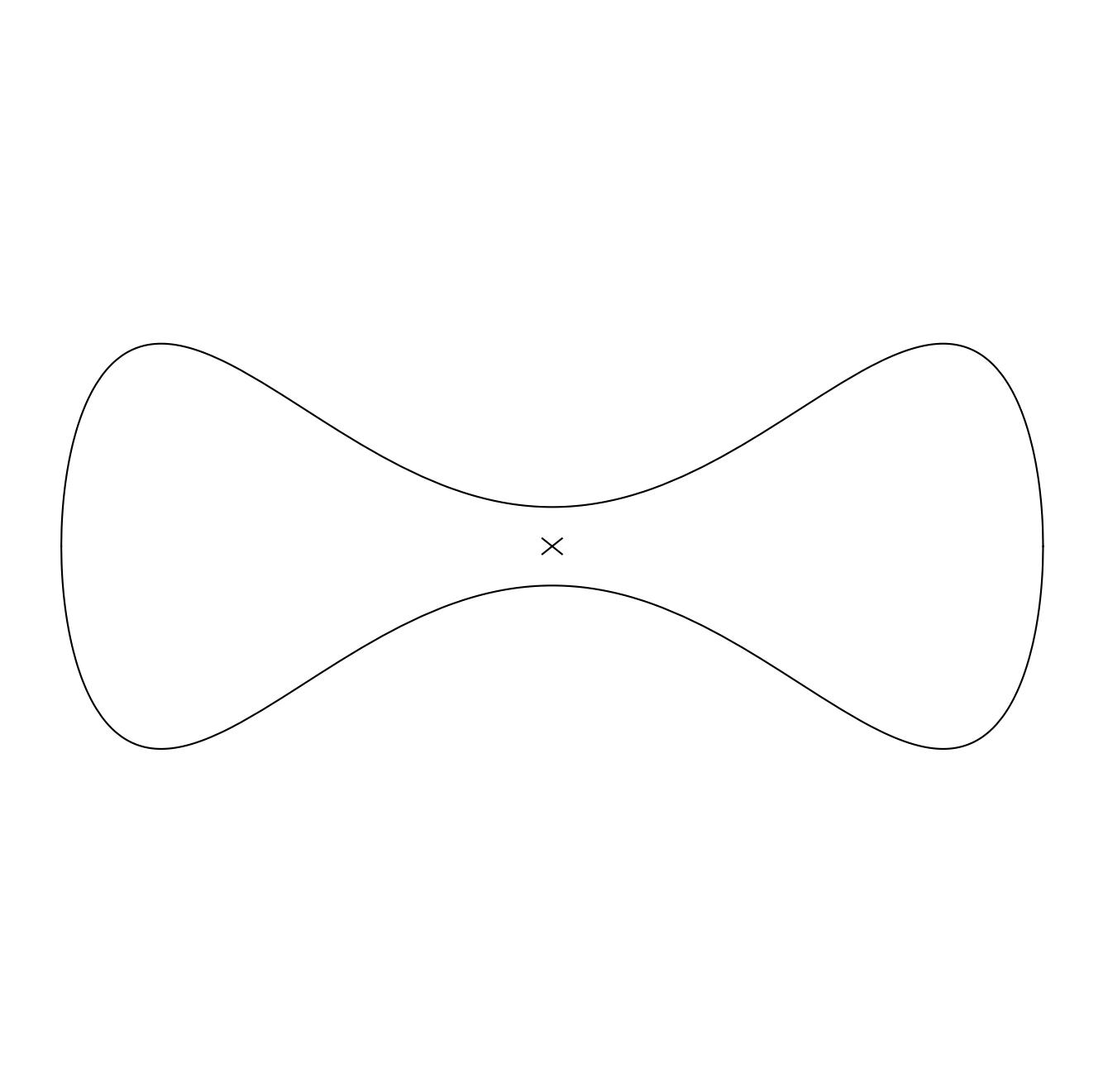}
		\end{subfigure}
\end{varwidth}}
\hspace{-7.8pt}
	\fbox{\begin{varwidth}{0.31\textwidth}
		\begin{subfigure}{\textwidth}
			
			\includegraphics[width=1\linewidth]{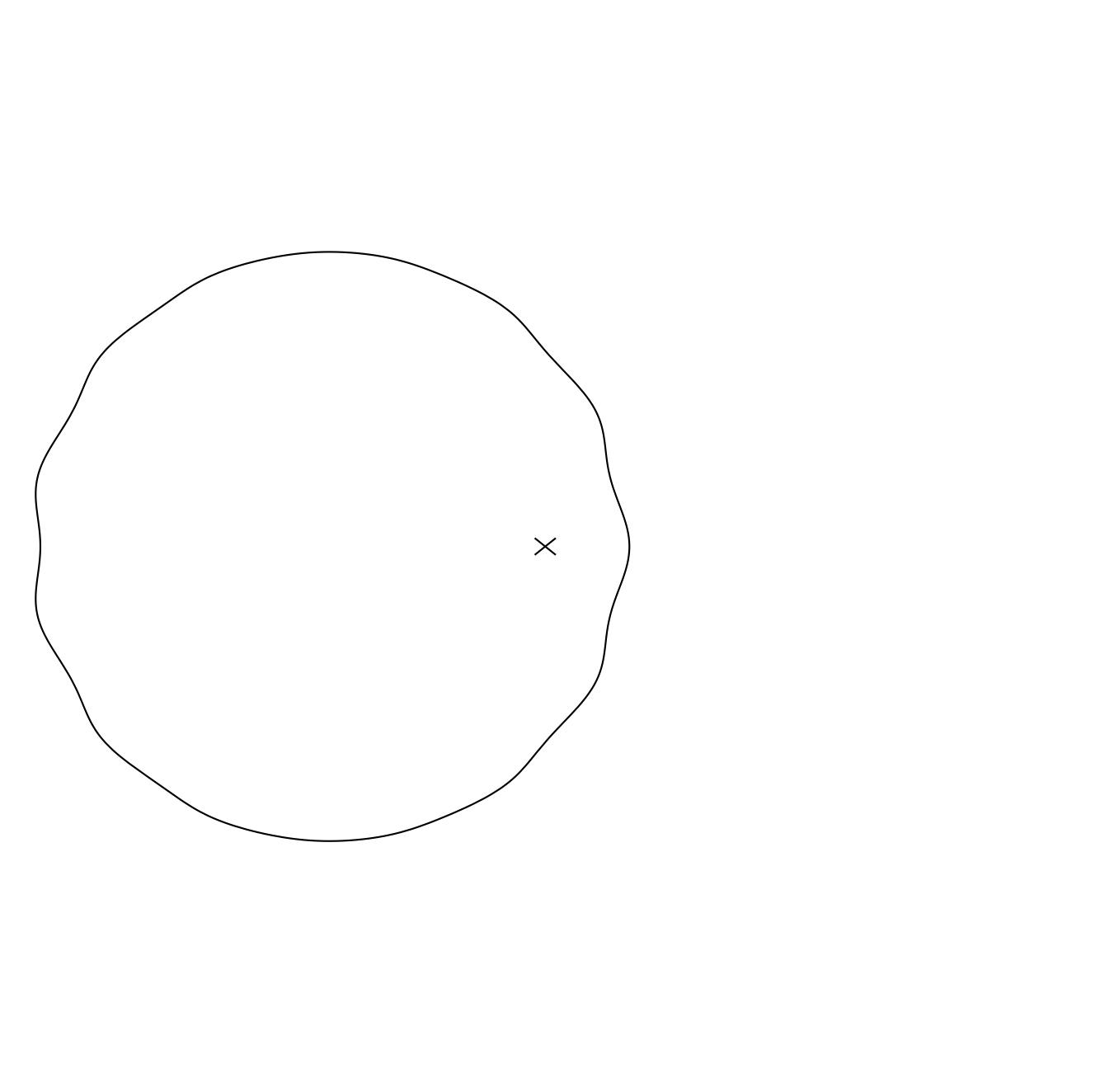}
		\end{subfigure}
\end{varwidth}}

	\caption{An illustration of the assumptions \eqref{required properties} in the uniqueness statement of Theorem \ref{general existence thm}. The cross indicates the origin in the asymptotically flat chart. The surface on the left violates the assumption $\rho(\Sigma)<4\,\lambda(\Sigma)$. The surface on the right violates the assumption $\lambda(\Sigma)<4\,\rho(\Sigma)$. The surface in the middle violates the small energy assumption.}
	\label{Figure alternatives}
\end{figure}
The assumptions on the scalar curvature in Theorem \ref{general existence thm} cannot be relaxed. On the one hand, the uniqueness statement fails if  assumption \eqref{growth with error 2} is dropped; see \cite[Theorem 13]{acws}. On the other hand, we show in the following that the existence of large area-constrained Willmore spheres with  comparable area radius and inner radius as well as small energy cannot be guaranteed if the scalar curvature is allowed to change signs.
\begin{thm} There exists a Riemannian $3$-manifold $(M,g)$ that is $C^k$-asymptotic to Schwarzschild with mass $m=2$ for every $k\geq 2$ with the following property. There exists no family $\{\Sigma(\kappa):\kappa\in(0,\kappa_0)\}$ of area-constrained Willmore spheres $\Sigma(\kappa)$ that enclose $B_2$ and  satisfy \eqref{constrained Willmore quantity}  with parameter $\kappa$ such that \label{existence counter}
$$
\lim_{\kappa\to0} \rho(\Sigma(\kappa))=\infty,\qquad 
\limsup_{\kappa\to0} \rho(\Sigma(\kappa))^{-1}\,\lambda(\Sigma(\kappa))<\infty,\qquad \text{and}\qquad \lim_{\kappa\to0} \int_{\Sigma(\kappa)}|\hcirc|^2\,\mathrm{d}\mu=0.
$$
\end{thm}
Theorem \ref{general existence thm} and Remark \ref{no far outlying} do not preclude the possibility of a sequence $\{\Sigma_i\}_{i=1}^\infty$ of large area-constrained Willmore spheres $\Sigma_i\subset M$ with
$$
\lim_{i\to\infty} \int_{\Sigma_i} |\hcirc|^2\,\mathrm{d}\mu=0
$$
that are slowly divergent in the sense that $$
\lim_{i\to\infty} \rho(\Sigma_i)=\infty\qquad\text{but}\qquad \lim_{i\to\infty} \rho(\Sigma_i)^{-1}\,\lambda(\Sigma_i)=\infty.
$$ 
As we discuss in \cite{acws}, it is a challenging analytical problem to rule out such sequences. Theorem \ref{slow divergence counter} below confirms that the existence of such a sequence hinges on the asymptotic behavior of the scalar curvature, too.   It should be compared with the uniqueness result obtained by J.~Qing and G.~Tian in \cite{qing2007uniqueness} for large stable constant mean curvature spheres.
\begin{thm}
	There exists a Riemannian $3$-manifold $(M,g)$ that is $C^k$-asymptotic to Schwarzschild with mass $m=2$ for every $k\geq 2$ such that the following holds. There exists a sequence $\{\Sigma_i\}_{i=1}^\infty$ of area-constrained Willmore spheres $\Sigma_i\subset M$ enclosing $B_2$ such that \label{slow divergence counter}
	$$
	\lim_{i\to\infty} \rho(\Sigma_i)=\infty,
	$$
	$\lambda(\Sigma_i)^{-1}\,\Sigma_i$ converges smoothly to a round sphere, but
	$$
	\lim_{i\to\infty} \rho(\Sigma_i)^{-1}\,\lambda_i(\Sigma_i)=\infty$$ and
	$m_H(\Sigma_i)>2$
	for every $i$.
	
\end{thm}

\subsection*{Outline of the paper} In order to prove Theorems \ref{acwt1}, \ref{com counterexample 2}, \ref{general existence thm}, \ref{existence counter}, and \ref{slow divergence counter}, we refine the Lyapunov-Schmidt analysis developed in our recent paper \cite{acws}. The method of Lyapunov-Schmidt analysis has previously been used by  S.~Brendle and the first-named author \cite{brendle2014large} and by O.~Chodosh and the first-named author \cite{chodoshfar} to study large stable constant mean curvature speres in Riemannian $3$-manifolds asymptotic to Schwarzschild. Contrary to the area-functional under a volume constraint, the Willmore energy under an area constraint  is translation invariant up to lower-order terms in exact Schwarzschild; see Lemma \ref{G lemma}. New difficulties owing to the competing contributions of the Schwarzschild background respectively the lower-order perturbation of the metric $\sigma$ arise when studying the center of mass of large area-constrained Willmore spheres.
 \\ \indent By scaling, we may assume throughout that $m=2$. Geometric computations are performed in the asymptotically flat chart \eqref{schwarzschild sigma decay}.  We use a bar to indicate that a geometric quantity has been computed with respect to the Euclidean background metric $\bar g$. When the Schwarzschild metric
$$
g_S=(1+|x|^{-1})^4\,\bar g
$$
with mass $m=2$ has been used in the computation, we use the subscript $S$.
\\ \indent 
Let $\delta\in(0,1/2)$. In \cite{acws}, we have used the implicit function theorem to construct surfaces $\Sigma_{\xi,\lambda}$ as perturbations of $S_\lambda(\lambda\,\xi)$ where $\xi\in\mathbb{R}^3$ with $|\xi|<1-\delta$ and $\lambda>1$ is large   such that $|\Sigma_{\xi,\lambda}|=4\,\pi\,\lambda^2$ and $\Sigma_{\xi,\lambda}$ is an area-constrained Willmore surface if and only if $\xi$ is a critical point of the function $G_\lambda$ defined by
$$G_\lambda(\xi)=\lambda^2\bigg(\int_{\Sigma_{\xi,\lambda}} H^2\,\mathrm{d}\mu-16\,\pi+64\,\pi\,\lambda^{-1}\bigg).
$$
We have also proven that 
$$
G_\lambda(\xi)=G_1(\xi)
+G_{2,\lambda}(\xi)+ G_{3,\lambda}(\xi)
$$
where $G_1$ is a rotationally symmetric and strictly convex function independent of $\lambda$, 
$$
G_{2,\lambda}(\xi)= 2\,\lambda\int_{\mathbb{R}^3\setminus{B_{\lambda}(\lambda\,\xi)}} R\,\mathrm{d}\bar{v},
$$
and $G_{3,\lambda}=O(\lambda^{-1})$ as $\lambda\to\infty$.  Here, $R$ is the scalar curvature of $(M,g)$. We refer to Appendix \ref{LS appendix} for more details on this construction. \\
\indent Under the assumptions of Theorem \ref{acw existence}, we have shown in \cite{acws} that the function $G_\lambda$ has a unique critical point $\xi(\lambda)\in\mathbb{R}^3$ with $\xi(\lambda)=o(1)$ as $\lambda\to\infty$. The sphere $\Sigma_{\xi(\lambda),\lambda}$ corresponds to a leaf  $\Sigma(\kappa)$ of the foliation \eqref{acw foliation} for suitable $\kappa=\kappa(\lambda)$. Moreover, we observe that
\begin{align*}
\lambda\,\xi(\lambda)=|\Sigma(\kappa(\lambda))|^{-1}\int_{\Sigma(\kappa(\lambda))} x^\ell\,\mathrm{d}\mu+O(\lambda^{-1}).
\end{align*}
On the one hand, we compute here that $\lambda\,(\bar DG_{3,\lambda})({\xi(\lambda)})$ is essentially proportional to the Hamiltonian center of mass $C$ provided $\lambda>1$ is sufficiently large. On the other hand, we prove that
$
\lambda\,(\bar D G_{2,\lambda})({\xi(\lambda)})
$
is small if the scalar curvature satisfies the assumptions of Theorem \ref{acwt1}. Since $\xi(\lambda)$ is a critical point of $G_\lambda$,  this proves Theorem \ref{acwt1}. By contrast, we show by explicit example  that $\lambda\,(\bar D G_{2,\lambda})({\xi(\lambda)})$ need not converge as $\lambda\to\infty$ if the assumptions on the scalar curvature are relaxed only slightly. This proves Theorem \ref{com counterexample 2}. \\
\indent To prove Theorem \ref{general existence thm}, we use a geometric argument to show that the function $G_{2,\lambda}$ is convex if the scalar curvature satisfies the growth condition \eqref{growth with error 2}. In particular, the function $G_\lambda$ has a critical point $\xi(\lambda)$  that is unique among all $\xi\in\mathbb{R}^3$ with $|\xi|<1-\delta$ provided $\lambda>1$ is sufficiently large. By contrast, we construct a metric whose scalar curvature changes signs such that for some sequence $\{\lambda_i\}_{i=1}^\infty$ with $\lim_{i\to\infty}\lambda_i=\infty$ and every $\delta>0$ there are infinitely many $i$ for which the function $G_{\lambda_i}$ has no critical points $\xi\in\mathbb{R}^3$ with $|\xi|<1-\delta$. Likewise, we construct a metric such that for every $\delta>0$ there are infinitely many $i$ for which $G_{\lambda_i}$ has a critical point $\xi_i \in\mathbb{R}^3$ with $1-\delta<|\xi_i|<1$. This proves Theorems \ref{existence counter} and \ref{slow divergence counter}.  
\subsection*{Acknowledgments} The authors would like to thank Otis Chodosh, Jan Metzger, and Felix Schulze for helpful discussions.  The authors acknowledge the support of the START-Project Y963 of the Austrian Science Fund (FWF). This version of the article has been accepted for publication after peer review
but is not the Version of Record and does not reflect post-acceptance improvements, or any
corrections. The Version of Record is available online at: \url{http://dx.doi.org/10.1007/s00220-022-04349-2}.
\section{Proof of Theorem \ref{acwt1} and Theorem \ref{com counterexample 2}}
 Throughout this section, we  assume that $(M,g)$ is $C^4$-asymptotic to Schwarzschild with mass $m=2$ and scalar curvature $R$ satisfying \label{com section}
\begin{align} \sum_{i=1}^3x^i\,\partial_i(| x|^2\, R)&\leq O(|x|^{-3}), \notag
\\  R(x)-R(-x)&=O(|x|^{-5}) \label{weaker center 2.2}.
\end{align}
\indent To prove Theorem \ref{acwt1}, we expand upon the Lyapunov-Schmidt analysis developed in our recent work \cite{acws}. The required concepts and estimates are summarized in Appendix \ref{LS appendix}. 
\\ \indent Let $\delta=1/4$, $\lambda_0>1$ be the constant from Proposition \ref{LS prop}, and $\xi\in\mathbb{R}^3$ with $|\xi|<3/4$. Recall the definitions \eqref{G definition}  of the function $G_\lambda$, \eqref{sigma xi lambda def} of the surface $\Sigma_{\xi,\lambda}$,  and \eqref{spheres def} of the sphere $S_{\xi,\lambda}$. Let $\xi(\lambda)$ be the unique critical point of $G_\lambda$ with
$
|\xi(\lambda)|<3/4
$ whose existence is asserted in Proposition \ref{existence prop}.\\ \indent Recall the Lagrange parameter $\kappa$ defined in \eqref{kappa function}. It follows from Proposition \ref{existence prop}, Proposition \ref{LS prop}, and Remark \ref{kappa remark} that 
$\Sigma_{\xi(\lambda),\lambda}
$ 
is the area-constrained Willmore sphere $\Sigma(\kappa)$ from \eqref{acw foliation} with $\kappa=\kappa(\Sigma(\lambda))$.
\begin{lem} There holds, as $\lambda\to\infty$,  \label{xi estimate}
	$$
	\xi(\lambda)=O(\lambda^{-1}).
	$$ 
\end{lem}
\begin{proof}
	Using Lemma \ref{G lemma}, Lemma \ref{g der lemma}, and that $(\bar DG_\lambda)({\xi(\lambda)})=0$, we find 
	$$
	0=|\xi(\lambda)|^{-1}\,\sum_{i=1}^3\xi(\lambda)^i(\partial_i G_\lambda)({\xi(\lambda)})\geq|\xi(\lambda)|^{-1}\,\sum_{i=1}^3\xi(\lambda)^i\,(\partial_i G_{1})({\xi(\lambda)})-O(\lambda^{-1}).
	$$ 
	Using \eqref{G1 definition}, we obtain
	$$
	|\xi(\lambda)|^{-1}\,\sum_{i=1}^3\xi(\lambda)^i\,(\partial_i G_{1})({\xi(\lambda)})\geq 256\,\pi\,|\xi(\lambda)|.
	$$ The assertion of the lemma follows from combining these estimates. 	
\end{proof}

Recall from Proposition \ref{LS prop} that $\Sigma_{\xi,\lambda}=\Sigma_{\xi,\lambda}(u_{\xi,\lambda})$, i.e.~$\Sigma_{\xi,\lambda}$  is the radial graph \eqref{radial graph} of the function $u_{\xi,\lambda}$ over $S_{\xi,\lambda}$.
We define
$$
\tilde u_{\xi,\lambda}=u_{\xi,\lambda}+2
$$
so that $$\Sigma_{\xi,\lambda}=\Sigma_{\tilde \xi,\tilde \lambda}(\tilde u_{\xi,\lambda})$$ 
with
\begin{align} \label{tilde normal relation}
\tilde \lambda = \lambda-2\qquad\text{and}\qquad  \tilde \xi=(\lambda-2)^{-1}\,\lambda\,\xi.
\end{align}
Note that $\lambda\,\xi=\tilde\lambda\,\tilde\xi$.\\ \indent 
Recall the vector field $Z_{\xi,\lambda}$ defined in \eqref{Z definition}. We abbreviate $u_{\xi,\lambda},\,\tilde u_{\xi,\lambda},\, Z_{\xi,\lambda},$ and $Z_{\tilde \xi,\tilde \lambda}$ by $u,\,\tilde u,\,Z$, and $\tilde Z$, respectively. Moreover, we let $\Lambda_0(S_{\tilde \xi,\tilde \lambda})\subset C^\infty(S_{\tilde \xi,\tilde \lambda})$ be the space of constant functions and $\Lambda_0^\perp(S_{\tilde \xi,\tilde \lambda})$ be its orthogonal complement. We abbreviate $\Lambda_0(S_{\tilde \xi,\tilde \lambda})$ by $\Lambda_0$
and $\Lambda_0^\perp(S_{\tilde \xi,\tilde \lambda})$ by $\Lambda_0^\perp$.
\begin{lem} There exists  $\delta\in(0,1/4)$ such that   \label{tilde u estimate}
	\begin{align} \label{tilde u estimate eq} 
	\tilde u=O(|\xi|^2)+O(\lambda^{-1})
	\end{align}
	and,  uniformly for every $\xi\in\mathbb{R}^3$ with $|\xi|<\delta $ as $\lambda\to\infty$,
	\begin{align} \label{tilde u proj}
	\operatorname{proj}_{\Lambda_0}\tilde u=-\lambda^{-1}-\frac{1}{16\,\pi}\,\lambda^{-1}\int_{S_{\xi,\lambda}}[\bar{\operatorname{tr}}\,\sigma-\sigma(\bar\nu,\bar\nu)]\,\mathrm{d}\bar\mu+O(\lambda^{-2})+O(\lambda^{-1}\,|\xi|^2).
	\end{align}
	 These expansions may be differentiated once with respect to $\xi$.
\end{lem}
\begin{proof}
	\eqref{tilde u estimate eq} follows directly from  \eqref{u estimate}. Using the identity
	$$
	\mathrm{d}\mu=\left[1+4\,|x|^{-1}+6\,|x|^{-2}+\frac12\,(\bar{\operatorname{tr}}\,\sigma-\sigma(\bar\nu,\bar\nu))+O(|x|^{-3})\right]\mathrm{d}\bar\mu,
	$$
\eqref{tilde normal relation}, and the fact that $(M,g)$ is $C^4$-asymptotic to Schwarzschild, we find that
$$
|S_{\tilde \xi,\tilde\lambda}|=4\,\pi\,\lambda^2+8\,\pi+\frac12 \int_{S_{\xi,\lambda}}[\bar{\operatorname{tr}}\,\sigma-\sigma(\bar\nu,\bar\nu)]\,\mathrm{d}\bar\mu+O(\lambda^{-1})+O(|\xi|^{2})
$$
provided that $\delta>0$ is sufficiently small. Using also that $|\Sigma_{\xi,\lambda}|=4\,\pi\,\lambda^2$, that $H(S_{\tilde\xi,\tilde\lambda})=2\,\lambda^{-1}+O(\lambda^{-2})$, and \eqref{tilde u estimate eq}, the first variation of area formula  therefore yields 
\begin{align*} 
2\,\lambda^{-1}\,\int_{S_{\tilde\xi,\tilde \lambda}}\tilde u\,\mathrm{d}\bar\mu=\,&|\Sigma_{\xi,\lambda}|-|S_{\tilde\xi,\tilde\lambda}|+O(\lambda^{-1})+O(|\xi|^{2})\\=\,&
-8\,\pi- \frac12 \int_{S_{\xi,\lambda}}[\bar{\operatorname{tr}}\,\sigma-\sigma(\bar\nu,\bar\nu)]\,\mathrm{d}\bar\mu+O(\lambda^{-1})+O(|\xi|^{2}).
\end{align*}
This implies \eqref{tilde u proj}.
\end{proof}
We proceed to compute a precise estimate for the Willmore energy of $\Sigma_{\xi,\lambda}$.
\begin{lem}  There exists $\delta\in(0,1/4)$ such that, \label{willmore expansion}  uniformly for every $\xi\in\mathbb{R}^3$ with $|\xi|<\delta $ as $\lambda\to\infty$,  
	\begin{align*} 
	\int_{\Sigma_{\xi,\lambda}}H^2\,\mathrm{d}\mu=&\,\int_{S_{\tilde\xi,\tilde\lambda}} H^2\,\mathrm{d}\mu-64\,\pi\,\lambda^{-3}-4\,\lambda^{-3}\,\int_{S_{\xi,\lambda}}[\bar{\operatorname{tr}}\,\sigma-\sigma(\bar\nu,\bar\nu)]\,\mathrm{d}\bar\mu\\&\qquad+O(\lambda^{-3}\,|\xi|^2)+O(\lambda^{-2}\,|\xi|^4\,)+O(\lambda^{-4}).
	\end{align*} 
	  This expansion may be differentiated once with respect to $\xi$.
\end{lem}

\begin{proof} According to Lemma \ref{willmore sphere}, we have
	$$
	\operatorname{proj}_{\Lambda_0}W(S_{\tilde\xi,\tilde\lambda})=-8\,\lambda^{-4}+O(\lambda^{-5})\qquad\text{and}\qquad 
\operatorname{proj}_{\Lambda_0^\perp}W(S_{\tilde \xi,\tilde\lambda})=O(\lambda^{-4}\,|\xi|^2)+O(\lambda^{-5}).
	$$
	The assertion follows from this, Lemma \ref{tilde u estimate} and Lemma \ref{willmore expansion lemma}. 
\end{proof}
\begin{rema} Let $\delta\in(0,1/4)$ and suppose that
	$$
	\mathcal{E}:\{\xi\in\mathbb{R}^3:|\xi|<\delta\}\times\{\lambda\in\mathbb{R}:\lambda>1\}\to\mathbb{R}
	$$ satisfies, as $\lambda\to\infty$,  $$\mathcal{E}=O(\lambda^{-3}\,|\xi|^2)+O(\lambda^{-2}\,|\xi|^4\,)+O(\lambda^{-4})$$ and
	$$
	\bar D\mathcal E=O(\lambda^{-3}\,|\xi|)+O(\lambda^{-2}\,|\xi|^3\,)+O(\lambda^{-4})
	$$
	where differentiation is with respect to $\xi$.
	Using Lemma \ref{xi estimate}, we find that, as $\lambda\to\infty$,
	$$
	(\bar D\mathcal{E})({\xi(\lambda)})=O(\lambda^{-4}).
	$$
\end{rema}

\begin{lem}	\label{schwarzschild contribution} There exists $\delta\in(0,1/4)$ such that,  uniformly for every $\xi\in\mathbb{R}^3$ with $|\xi|<\delta $ as $\lambda\to\infty$,
	$$
	\int_{S_{\tilde \xi,\tilde\lambda}} H_S^2\,\mathrm{d}\mu_S=16\,\pi-64\,\pi\,\lambda^{-1}+128\,\pi\,|\xi|^2\,\lambda^{-2}+O(\lambda^{-3}\,|\xi|^2)+O(\lambda^{-4}).
	$$
  	This expansion may be differentiated once with respect to $\xi$.
\end{lem}
\begin{proof}
	This follows from \eqref{tilde normal relation} and a direct computation similar to that in \cite[Lemma 42]{acws}.
\end{proof}
Recall the conformal Killing operator $\mathcal{D}$ defined in \eqref{killing definition}.
\begin{lem}  There exists $\delta\in(0,1/4)$ such that, 	\label{willmore linearization}  uniformly for every $\xi\in\mathbb{R}^3$ with $|\xi|<\delta $ as $\lambda\to\infty$,
	\begin{align*}
	\int_{S_{\tilde \xi,\tilde \lambda}} H^2\,\mathrm{d}\mu&-\int_{S_{\tilde \xi,\tilde \lambda}} H_S^2\,\mathrm{d}\mu_S
	\\=\,&2\,{\tilde \lambda}^{-1}\int_{\mathbb{R}^3\setminus B_{\tilde\lambda}(\lambda\,\xi)} R\,\mathrm{d}\bar v 	
	\\\,&\qquad+8\,\lambda^{-3}\int_{S_{\xi,\lambda}}\bigg[(\sigma(\bar\nu,\bar\nu)-6\,\bar g(\xi,\bar\nu))\,\sigma(\bar\nu,\bar\nu)+3\,\sigma(\bar\nu,\xi)\bigg]\,\mathrm{d}\bar\mu
	\\\,&\qquad -4\,\lambda^{-1}\,\int_{\mathbb{R}^3\setminus B_\lambda(\lambda\,\xi)}|x|^{-3} \,\sum_{i=1}^3\bigg[2\,(\bar D^2_{e_i,x}\sigma)(e_i,x)-(\bar D^2_{e_i,e_i}\sigma)(x,x)-(\bar D^2_{x,x}\sigma)(e_i,e_i)\\&\qquad \qquad\qquad\qquad\qquad\qquad\qquad\quad
	+\sum_{j=1}^3\big[(\bar D^2_{e_i,e_i}\sigma)(e_j,e_j)-(\bar D^2_{e_i,e_j}\sigma)(e_i,e_j)\big]\bigg]\,\mathrm{d}\bar\mu
	\\\,&\qquad +4\int_{\mathbb{R}^3\setminus B_\lambda(\lambda\,\xi)}|x|^{-3} \,\sum_{i=1}^3\,\bigg[(\bar D^2_{e_i,x}\sigma)(e_i,\xi)+ (\bar D^2_{e_i,\xi}\sigma)(e_i,x)
	\\&\qquad\qquad\qquad\qquad\qquad\qquad\qquad-(\bar D^2_{e_i,e_i}\sigma)(x,\xi)-(\bar D^2_{x,\xi}\sigma)(e_i,e_i)\bigg]\,\mathrm{d}\bar\mu
	\\\,&\qquad -4\int_{\mathbb{R}^3\setminus B_\lambda(\lambda\,\xi)} 
	|x|^{-3}\bigg[\lambda^{-1}\bar D_x\bar{\operatorname{tr}}\,\sigma-3\,\lambda^{-1}\,|x|^{-2}\,(\bar D_x\sigma)(x,x)-\bar D_\xi\bar{\operatorname{tr}}\,\sigma\\&\qquad \qquad\qquad\qquad\qquad\qquad+3\,|x|^{-2}\,(\bar D_\xi\sigma)(x,x)\bigg]
	\,\mathrm{d}\bar\mu
	\\\,&\qquad +O(\lambda^{-4})+O(|\xi|^2\,\lambda^{-3}).
	\end{align*}
	 This expansion may be differentiated once with respect to $\xi$.
\end{lem}
\begin{proof}
	As in the proof of \cite[Lemma 42]{acws}, there holds
	\begin{equation} \label{w 0}
	\begin{aligned}
	\int_{S_{\tilde \xi,\tilde \lambda}} H^2\,\mathrm{d}\mu=&16\,\pi-64\,\pi\,\tilde\lambda^{-1}+2\,\int_{S_{\tilde\xi,\tilde\lambda}} |\hcirc|^2\,\mathrm{d}\mu+\frac{2}{3}\int_{\mathbb{R}^3\setminus B_{\tilde \lambda} (\lambda\,\xi)}(\operatorname{div} \tilde Z)\, R\,
	\mathrm{d}v
	\\&+4\int_{S_{\tilde \xi,\tilde \lambda}} \operatorname{Ric}(\nu-\tilde Z,\nu)\,\mathrm{d}\mu
	-2\int_{\mathbb{R}^3\setminus B_{\tilde \lambda}(\lambda\,\xi)} g(\operatorname{Ric},\mathcal{D}\tilde Z)\,\mathrm{d}v.
	\end{aligned}
	\end{equation}

\indent Note that the first integral on the right-hand side is conformally invariant. It follows that
\begin{align}
 \label{w 1a}
\int_{S_{\tilde\xi,\tilde\lambda}} |\hcirc_S|^2\,\mathrm{d}\mu_S=0.
\end{align}
Likewise, using $R_S=0$, we find that
\begin{align} \label{w 1b} \int_{\mathbb{R}^3\setminus B_{\tilde \lambda} (\lambda\,\xi)}
	\operatorname{div}_S  \tilde Z\,R_S\,\mathrm{d}v_S
	=0.	
	\end{align}
	Similarly, we have $\tilde Z=\nu_S$ on $S_{\tilde\xi,\tilde\lambda}$ and consequently
	\begin{align} \label{w 1c}
	\int_{S_{\tilde \xi,\tilde \lambda}} \operatorname{Ric}_S(\nu_S-\tilde Z,\nu_S)\,\mathrm{d}\mu_S=0.
	\end{align}
\indent 	According to Lemma \ref{hcirc estimate}, there holds
	\begin{align} \label{w 2}
	\int_{S_{\tilde \xi,\tilde \lambda}} |\hcirc|^2\,\mathrm{d}\mu=O(\lambda^{-4}).
	\end{align}
\indent 	We compute that
	$$
	\operatorname{div}\tilde Z=3\,(1+|x|^{-1})^{-2}\,\tilde \lambda^{-1}-4\,\tilde \lambda^{-1}\,|x|^{-1}+4\,|x|^{-3}\,\bar g(\tilde \xi,x)+O(\lambda^{-1}\,|x|^{-2})+O(|x|^{-3}).
	$$
	In conjunction with the estimate $$\mathrm{d}\mu=[1+4\,|x|^{-1} +O(|x|^{-2})]\,\mathrm{d}\bar\mu,$$ it follows that
	\begin{equation}  \label{w 3}
	\begin{aligned}
	\int_{\mathbb{R}^3\setminus B_{\tilde \lambda} (\lambda\,\xi)} \operatorname{div} \tilde Z\,R\,
	\mathrm{d}v=\,&\tilde \lambda^{-1} \int_{\mathbb{R}^3\setminus B_{\tilde \lambda} (\lambda\,\xi)}
	\bigg[3+2\,|x|^{-1}+4\,|x|^{-3}\,\bar g(\xi,x)\bigg]\,R\,\mathrm{d}\bar v\\&\qquad +O(\lambda^{-4})+O(|\xi|^2\,\lambda^{-3}).
	\end{aligned} 
	\end{equation}
\indent	Using \eqref{tilde normal relation}, \eqref{rics}, and \eqref{nu vs nus}, we find 
	\begin{equation} \label{w 4}
	\begin{aligned}
	\int_{S_{\tilde \xi,\tilde \lambda}} \operatorname{Ric}(\nu-\tilde Z,\nu)\,\mathrm{d}\mu=\,&\int_{S_{\xi,\lambda}} \operatorname{Ric}_S(\nu-Z,\bar\nu)\,\mathrm{d}\bar \mu+O(\lambda^{-4})
	\\=\,&2\,\lambda^{-3}\int_{S_{\xi,\lambda}}\bigg[\sigma(\bar\nu,\bar\nu)-6\,\bar g(\xi,\bar\nu)\,\sigma(\bar\nu,\bar\nu)+3\,\sigma(\bar\nu,\xi)\bigg]\,\mathrm{d}\bar\mu\\&\qquad +O(\lambda^{-4})+O(|\xi|^2\,\lambda^{-3}).
	\end{aligned}
	\end{equation}
	\indent Using \eqref{tilde normal relation}, we obtain
	\begin{equation} 
	 \label{w 5}
	\begin{aligned}
	\int_{\mathbb{R}^3\setminus B_{\tilde \lambda}(\lambda\,\xi)}& g(\operatorname{Ric},\mathcal{D}\tilde Z)\,\mathrm{d}\mu-\int_{\mathbb{R}^3\setminus B_{\tilde \lambda}(\lambda\,\xi)} g_S(\operatorname{Ric}_S,\mathcal{D}_S\tilde Z)\,\mathrm{d}\mu_S\\=\,&\int_{\mathbb{R}^3\setminus B_{\tilde \lambda}(\lambda\,\xi)} \bigg[\bar g({\operatorname{Ric}}-\operatorname{Ric}_S,\mathcal{D}_S\tilde Z)+\bar g({\operatorname{Ric}_S},\mathcal{D}\tilde Z-\mathcal{D}_S\tilde Z)\bigg]\,\mathrm{d}\bar\mu
	+O(\lambda^{-4})
	\\=\,&\int_{\mathbb{R}^3\setminus B_{ \lambda}(\lambda\,\xi)} \bigg[\bar g({\operatorname{Ric}}-\operatorname{Ric}_S,\mathcal{D}_SZ)+\bar g({\operatorname{Ric}_S},\mathcal{D} Z-\mathcal{D}_S Z)\bigg]\,\mathrm{d}\bar\mu+O(\lambda^{-4}).
	\end{aligned}
	\end{equation} 
	Using \eqref{ricci expansion equation}, \eqref{schwarzschild killing}, and that $R_S=0$, we compute
	\begin{equation} 
	\begin{aligned}
	\int_{\mathbb{R}^3\setminus B_\lambda(\lambda\,\xi)} &\bar g({\operatorname{Ric}-\operatorname{Ric}_S},\mathcal{D}_SZ)\,\mathrm{d}\bar\mu\\=&\,2\,\lambda^{-1}\int_{\mathbb{R}^3\setminus B_\lambda(\lambda\,\xi)}|x|^{-3}\,\sum_{i=1}^3 \bigg[2\,(\bar D^2_{e_i,x}\sigma)(e_i,x)-(\bar D^2_{e_i,e_i}\sigma)(x,x)- (\bar D^2_{x,x}\sigma)(e_i,e_i)\bigg]\,\mathrm{d}\bar\mu\\
	&\qquad \,-2\int_{\mathbb{R}^3\setminus B_\lambda(\lambda\,\xi)}|x|^{-3}\,\sum_{i=1}^3 \bigg[(\bar D^2_{e_i,x}\sigma)(e_i,\xi)+(\bar D^2_{e_i,\xi}\sigma)(e_i,x)
	\\&\qquad\qquad\qquad\qquad\qquad\qquad\qquad-(\bar D^2_{e_i,e_i}\sigma)(x,\xi)-(\bar D^2_{x,\xi}\sigma)(e_i,e_i)\bigg]\,\mathrm{d}\bar\mu
	\\&\qquad \,+\frac43\int_{\mathbb{R}^3\setminus B_{\lambda}(\lambda\,\xi)}\bigg[|x|^{-3}\,\bar g(x,\xi)-\lambda^{-1}\,|x|^{-1}\bigg]\,R\,\mathrm{d}\bar\mu \\	
	&\qquad \,+O(\lambda^{-4}).
	\end{aligned} \label{w 6}
	\end{equation}
	For the first line of the right-hand side of \eqref{w 6}, we note that  
	\begin{align} 
	R=\sum_{i,\,j=1}^3\big[(\bar D^2_{e_i,e_j}\sigma)(e_i,e_j)-\bar D^2_{e_i,e_i}\sigma(e_j,e_j)\big]+O(|x|^{-5});
	\end{align} 
	see \eqref{ricci expansion equation}.
	Finally, using \eqref{Killing changes} and that $R_S=0$, we obtain
	\begin{equation} \label{w 7}
	\begin{aligned}
	\int_{\mathbb{R}^3\setminus B_\lambda(\lambda\,\xi)}& \bar g({\operatorname{Ric}_S},{\mathcal{D}Z}-\mathcal{D}_SZ)\,\mathrm{d}\bar\mu\\=&\,2\int_{\mathbb{R}^3\setminus B_\lambda(\lambda\,\xi)} 
	|x|^{-3}\bigg[\lambda^{-1}(\bar D_x\bar{\operatorname{tr}}\,\sigma)-3\,\lambda^{-1}\,|x|^{-2}\,(\bar D_x\sigma)(x,x)-(\bar D_\xi\bar{\operatorname{tr}}\,\sigma)\\&\qquad\qquad\qquad\qquad+3\,|x|^{-2}\,(\bar D_\xi\sigma)(x,x)\bigg]
	\,\mathrm{d}\bar\mu
	\\&\qquad +O(\lambda^{-4}).
	\end{aligned}
	\end{equation}
	\indent Assembling (\ref{w 0}-\ref{w 7}), the assertion of the lemma follows.
\end{proof}
\begin{prop}
	There holds, as $\lambda\to\infty$, \label{der prop}
	\begin{align*}
	256\,\pi\,\lambda\,\xi(\lambda)=&\,2\,\lambda^3\,\int_{S_{\lambda}(\lambda\,\xi(\lambda))}R\,\bar \nu\,\mathrm{d}\bar\mu\\&\qquad +
	8\,\lambda\int_{S_\lambda(0)}\left[(\bar{D}\,\bar{\operatorname{tr}}\,\sigma)-(\bar D\sigma)(\bar\nu,\bar\nu)-2\,\lambda^{-1}\,\bar{\operatorname{tr}}\,\sigma\,\bar\nu\right]\,\mathrm{d}\bar\mu\\&\qquad +O(\lambda^{-1}).
	\end{align*}
\end{prop}
\begin{proof}
Recall that
	$
	(\bar D G_\lambda)({\xi(\lambda)})=0.
	$
In conjunction with Lemma \ref{xi estimate} and Lemma \ref{willmore expansion}, this implies
	\begin{align} \label{p1}
	\left(\bar D \int_{S_{\xi,\tilde \lambda}} H^2\,\mathrm{d}\mu\right)({\tilde \xi(\lambda)})-4\,\lambda^{-2}\,\int_{S_{\xi,\lambda}}\big[\bar D\,\bar{\operatorname{tr}}\,\sigma-(\bar D\sigma)(\bar\nu,\bar\nu)\big]\,\mathrm{d}\bar\mu+O(\lambda^{-4})=0.
	\end{align}
 Lemma \ref{xi estimate} and Lemma \ref{schwarzschild contribution} imply that
	\begin{align} \label{p2}
	\left(\bar D \int_{S_{\xi,\tilde \lambda}} H_S^2\,\mathrm{d}\mu_S\right)({\tilde \xi(\lambda)})=256\,\pi\,\lambda^{-2}\,\xi(\lambda)+O(\lambda^{-4}).
	\end{align}
	Using \eqref{tilde normal relation}, Lemma \ref{xi estimate}, and the fact that $(M,g)$ is $C^4$-asymptotic to Schwarzschild, we infer from Lemma \ref{willmore linearization} that, for every $a\in\mathbb{R}^3$ with $|a|=1$, we have
	\begin{equation} \label{p3}
	\begin{aligned}
	&\left(\bar D_a \int_{S_{ \xi,\tilde \lambda}} H^2\,\mathrm{d}\mu\right)({\tilde \xi(\lambda)})-\left(\bar D_a\int_{S_{\xi,\tilde \lambda}} H_S^2\,\mathrm{d}\mu_S\right)({\tilde \xi(\lambda)})
	\\&\qquad =\,-2 \,\lambda\,\tilde \lambda\,\int_{S_{\tilde \lambda}(\lambda\, \xi(\lambda))}\bar g(\bar \nu,a)\,R\,\mathrm{d}\bar\mu
	\\
	&\qquad\qquad +\,8\,\lambda^{-3}\int_{S_{\lambda}(0)}\bigg[\lambda\,(\bar D_a\sigma)(\bar\nu,\bar\nu)-6\,\bar g(a,\bar\nu)\,\sigma(\bar\nu,\bar\nu)+3\,\sigma(\bar\nu,a)\bigg]\,\mathrm{d}\bar\mu
	\\\,&\qquad\qquad+4\,\lambda^{-1}\,\int_{S_\lambda(0)}\bar g(\nu,a)\,\sum_{i=1}^3\bigg[2\,(\bar D^2_{e_i,\bar\nu}\sigma)(\bar\nu,e_i)-(\bar D^2_{e_i,e_i}\sigma)(\bar\nu,\bar\nu)-(\bar D^2_{\bar\nu,\bar\nu}\sigma)(e_i,e_i)\\&\qquad\qquad\qquad\qquad\qquad\qquad\qquad\qquad +\sum_{j=1}^3\big[(\bar D^2_{e_i,e_i}\sigma)(e_j,e_j)-(\bar D^2_{e_i,e_j}\sigma)(e_i,e_j)\big]\bigg]\,\mathrm{d}\bar\mu
	\\\,&\qquad\qquad+4\,\int_{\mathbb{R}^3\setminus B_\lambda(0)}|x|^{-3} \,\sum_{i=1}^3\bigg[(\bar D^2_{e_i,x}\sigma)(e_i,a)+(\bar D^2_{e_i,a}\sigma)(e_i,x)\\&\qquad\qquad\qquad\qquad\qquad\qquad\qquad\quad -(\bar D^2_{e_i,e_i}\sigma)(x,a)-(\bar D^2_{x,a}\sigma)(e_i,e_i)\bigg]\,\mathrm{d}\bar\mu
	\\\,&\qquad\qquad+4\,\lambda^{-2}\,\int_{S_\lambda(0)} \bar g(\bar\nu,a)\, 
	\bigg[\bar D_{\bar\nu}\bar{\operatorname{tr}}\,\sigma-3\,(\bar D_{\bar\nu}\sigma)(\bar\nu,\bar\nu)\bigg]\,\mathrm{d}\bar\mu
	\\\,&\qquad\qquad+4\,\int_{\mathbb{R}^3\setminus B_\lambda(0)} |x|^{-3}\,
	\bigg[\bar D_a\bar{\operatorname{tr}}\,\sigma-3\,|x|^{-2}\,(\bar D_a\sigma)(x,x)\bigg]
	\,\mathrm{d}\bar\mu
	\\\,&\qquad\qquad+O(\lambda^{-4}).
	\end{aligned}
	\end{equation}
\indent 

	Since $(M,g)$ is $C^4$-asymptotic to Schwarzschild, we obtain from (\ref{weaker center 2.2})  that
	\begin{align}
\sum_{i=1}^3	x^i\big[(\partial_i R)(x)+(\partial_i R)(-x)\big]=O(|x|^{-5}).
	\label{dweakercenter}
	\end{align}
	Indeed, if (\ref{dweakercenter}) failed, integration along radial lines would imply that (\ref{weaker center 2.2}) fails as well.
	Consequently, in conjunction with \eqref{tilde normal relation} and Lemma \ref{xi estimate},  we obtain 
	\begin{equation} \label{p5}
	\int_{S_{\tilde \lambda}(\lambda\, \xi(\lambda))}\,\bar g(\bar \nu,a)\,R\,\mathrm{d}\bar\mu=\int_{S_{\lambda}(\lambda\, \xi(\lambda))}\,\bar g(\bar \nu,a)\,R\,\mathrm{d}\bar\mu+O(\lambda^{-1}).
	\end{equation}
	\indent Next, we observe that all derivatives of the form $\bar D^2_{\bar\nu,\bar\nu}$ in the term
	$$
\sum_{i=1}^3\bigg[	2\,(\bar D^2_{e_i,\bar\nu}\sigma)(\bar\nu,e_i)-(\bar D^2_{e_i,e_i}\sigma)(\bar\nu,\bar\nu)-(\bar D^2_{\bar\nu,\bar\nu}\sigma)(e_i,e_i)+\sum_{j=1}^3\big[(\bar D^2_{e_i,e_i}\sigma)(e_j,e_j)-(\bar D^2_{e_i,e_j}\sigma)(e_i,e_j)\big]\bigg]
	$$cancel. We may therefore use integration by parts and the decomposition $a=a^\perp+a^\top$ with respect to $\bar g$ to find that
	\begin{equation}
	\begin{aligned}
	&\int_{S_\lambda(0)}\bar g(\nu,a)\,\sum_{i=1}^3\bigg[	2\,(\bar D^2_{e_i,\bar\nu}\sigma)(\bar\nu,e_i)-(\bar D^2_{e_i,e_i}\sigma)(\bar\nu,\bar\nu)-(\bar D^2_{\bar\nu,\bar\nu}\sigma)(e_i,e_i)\\&\qquad\qquad\qquad\qquad\quad +\sum_{j=1}^3\big[(\bar D^2_{e_i,e_i}\sigma)(e_j,e_j)-(\bar D^2_{e_i,e_j}\sigma)(e_i,e_j)\big]\bigg]
	\,\mathrm{d}\bar\mu
	\\ &\qquad =\lambda^{-1}\int_{S_\lambda(0)}\bigg[\bar D_a\bar{\operatorname{tr}}\,\sigma-(\bar D_a\sigma)(\bar\nu,\bar\nu)-2\,\lambda^{-1}\,\bar{\operatorname{tr}}\,\sigma\bigg]\,\mathrm{d}\bar\mu.
	\end{aligned}	
	\end{equation}
\indent 	Next, note that the vector field $Y=|x|^{-3}\,x$ is divergence free. Let
	$$
	T=\sum_{i=1}^3\bigg[(\bar D_Y\sigma)(a,e_i)+(\bar D_a\sigma)(Y,e_i)-(\bar D_{e_i}\sigma)(Y,a)-\bar g(Y,e_i)\,\bar D_a\bar{\operatorname{tr}}\,\sigma\bigg]e_i.
	$$
	There holds
	\begin{align*}
	\bar {\operatorname{div}}\,T=&|x|^{-3}\,\sum_{i=1}^3 \bigg[(\bar D^2_{e_i,x}\sigma)(e_i,a)+(\bar D^2_{e_i,a}\sigma)(e_i,x)-(\bar D^2_{e_i,e_i}\sigma)(x,a)-(\bar D^2_{x,a}\sigma)(e_i,e_i)\bigg]\\&\qquad  +|x|^{-3}\,\big[\bar D_a\bar{\operatorname{tr}}\,\sigma-3\,|x|^{-2}\,(\bar D_a\sigma)(x,x)\big].
	\end{align*}
	Consequently,
		\begin{equation} \label{p60}
	\begin{aligned}
&\int_{\mathbb{R}^3\setminus B_\lambda(0)}|x|^{-3} \sum_{i=1}^3\bigg[(\bar D^2_{e_i,x}\sigma)(e_i,a)+(\bar D^2_{e_i,a}\sigma)(e_i,x)-(\bar D^2_{e_i,e_i}\sigma)(x,a)-(\bar D^2_{x,a}\sigma)(e_i,e_i)\bigg]\,\mathrm{d}\bar\mu	
	\\\,&\qquad+\int_{\mathbb{R}^3\setminus B_\lambda(0)} |x|^{-3}\,
	\bigg[\bar D_a\bar{\operatorname{tr}}\,\sigma-3\,|x|^{-2}\,(\bar D_a\sigma)(x,x)\bigg]
	\,\mathrm{d}\bar\mu
	\\&\qquad\qquad=\,\lambda^{-2}\int_{S_\lambda(0)} \bigg[\bar D_a\bar{\operatorname{tr}}\,\sigma-(\bar D_a\sigma)(\bar\nu,\bar\nu)\bigg]\,\mathrm{d}\bar\mu.
	\end{aligned}
	\end{equation}
\indent Finally, we use integration by parts and the decomposition $a=a^\perp+a^\top$ to find that
	\begin{equation} \label{p6}
	\begin{aligned}
	&2\, \lambda^{-3}\int_{S_{\lambda}(0)}\bigg[\lambda\,(\bar D_a\sigma)(\bar\nu,\bar\nu)-6\,\bar g(a,\bar\nu)\,\sigma(\bar\nu,\bar\nu)+3\,\sigma(\bar\nu,a)\bigg]\,\mathrm{d}\bar\mu
	\\\,&\qquad+\,\lambda^{-2}\int_{S_\lambda(0)} \bar g(\bar\nu,a)\, 
	\bigg[\bar D_{\bar\nu}\bar{\operatorname{tr}}\,\sigma-3\,(\bar D_{\bar\nu}\sigma)(\bar\nu,\bar\nu)\bigg]\,\mathrm{d}\bar\mu
	\\&\qquad\qquad  = \lambda^{-2}\int_{S_\lambda(0)}\bigg[\bar{D}_a\bar{\operatorname{tr}}\,\sigma-(\bar D_a\sigma)(\bar\nu,\bar\nu)-2\,g(a,\bar\nu)\,\bar{\operatorname{tr}}\,\sigma\bigg]\,\mathrm{d}\bar\mu
	\\&\qquad\qquad\qquad +O(\lambda^{-4}).
	\end{aligned}
	\end{equation}
\indent The assertion follows from assembling (\ref{p1}-\ref{p6}).
\end{proof} For the corollary below, recall the definition \eqref{center of mass} of the Hamiltonian center  $C=(C^1,\,C^2,\,C^3)$ of $(M,g)$. 
\begin{coro}
	There holds, as $\lambda\to\infty$, \label{com coro}
	$$
	\lambda\,\xi(\lambda)=C+\frac{1}{128\,\pi}\,\lambda^3\,\int_{S_{\lambda}(\lambda\,\xi(\lambda))}R\,\bar \nu\,\mathrm{d}\bar\mu+o(1).
	$$ 
\end{coro}
\begin{proof}
	We define the quantities
	\begin{equation*}
	\begin{aligned}
		z^\ell =\,&\frac{1}{32\,\pi}\,\lambda^{-1}\,\int_{S_\lambda(0)}\bigg(\,\sum_{i,\,j=1}^3x^\ell\,x^j\,\bigg[(\partial_i\sigma)(e_i,e_j)-(\partial_j\sigma)(e_i,e_i)\bigg]
		\\& \qquad\qquad\qquad\qquad-\sum_{i=1}^3\bigg[x^i\,\sigma(e_i,e_\ell)-x^\ell\,\sigma(e_i,e_i)\bigg]\bigg)\,\mathrm{d}\bar\mu
	\end{aligned}
\end{equation*}
where $\ell=1,\,2,\,3$. Note that, by \eqref{center of mass}, 
\begin{equation} \label{lim}
\lim_{\lambda\to\infty}z^\ell=C^\ell.
\end{equation} Using integration by parts and the decomposition $e_\ell=e_\ell^\perp +e_\ell^\top$ with respect to $\bar g$, we obtain 	\begin{equation} \label{lim2}
	z^\ell=\frac{1}{32\,\pi}\,\lambda\int_{S_\lambda(0)}\bigg[( \partial_{\ell}\sigma)(\bar\nu,\nu)- \partial_{\ell}\bar{\operatorname{tr}}\,\sigma+2\,\lambda^{-1}\,\bar\nu^\ell\,\bar{\operatorname{tr}}\,\sigma\bigg]\mathrm{d}\mu.
	\end{equation} 
	The assertion follows from \eqref{lim}, \eqref{lim2}, and Proposition \ref{der prop}.
\end{proof}
\begin{proof}[Proof of Theorem \ref{acwt1}] 
	Corollary \ref{com coro} implies that
	$$
	|\lambda\,\xi(\lambda)-C|\leq \frac{1}{128\,\pi}\,\lambda^3\,|\lambda\,\xi(\lambda)-C|^{-1}\int_{S_{\lambda}(\lambda\,\xi(\lambda))}R\,\bar g(\lambda\,\xi(\lambda)-C,\bar\nu)\,\mathrm{d}\bar\mu+o(1).
	$$
	Arguing as in the proof of Lemma \ref{g der lemma} but using the stronger asymptotic conditions \eqref{improved center}, we obtain 
	$$
	\frac{1}{128\,\pi}\,\lambda^3\,|\lambda\,\xi(\lambda)-C|^{-1}\int_{S_{\lambda}(\lambda\,\xi(\lambda))}R\,\bar g(\lambda\,\xi(\lambda)-C,\bar\nu)\,\mathrm{d}\bar\mu\leq o(1).
	$$
	Conversely, \eqref{tilde u estimate} and Lemma \ref{xi estimate} imply that
	\begin{align} 
	|\Sigma_{\xi(\lambda),\lambda}|^{-1}\int_{\Sigma_{\xi(\lambda),\lambda}}x^\ell\,\mathrm{d}\mu=\lambda\,\xi(\lambda)+O(\lambda^{-1}).
	\label{com relation}
	\end{align}
	The assertion  follows from these estimates.
\end{proof}
	\begin{figure}\centering
	\includegraphics[width=0.5\linewidth]{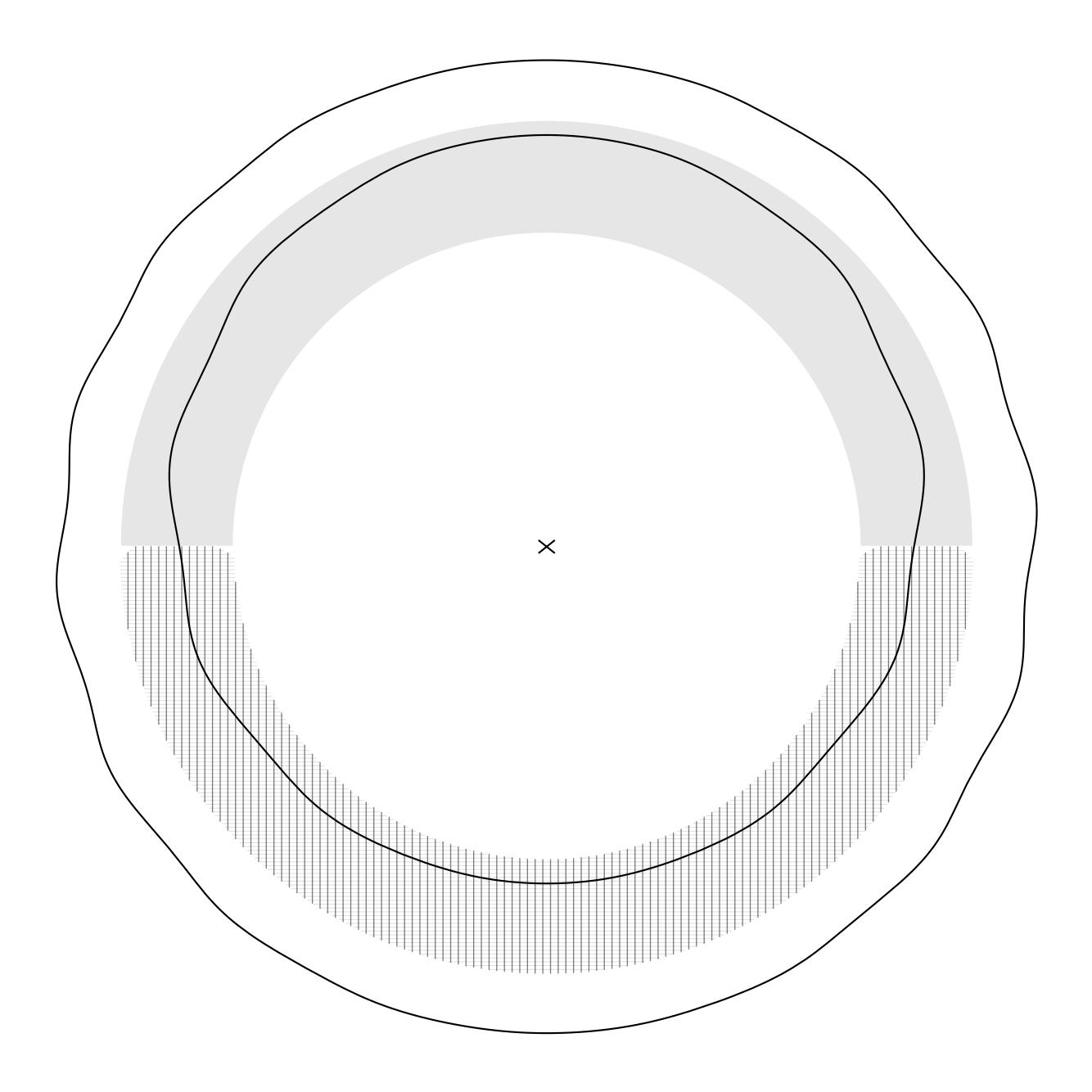}
	\caption{An illustration of the proof of Theorem \ref{com counterexample 2}. The
		 scalar curvature is positive in the shaded region, negative in the
		hatched region, and vanishes elsewhere. The cross marks the Hamiltonian center of mass $C$ in the asymptotically flat chart. The barycenter of the larger sphere  $\Sigma_{\hat\lambda_k,\xi(\hat\lambda_k)}$ agrees with $C$. By contrast, the asymmetric distribution of scalar curvature moves the  barycenter of the smaller sphere $\Sigma_{\lambda_k,\xi(\lambda_k)}$ away from $C$. }
	\label{comccounterfigure}
\end{figure}
\begin{proof}[Proof of Theorem \ref{com counterexample 2}] Let $\chi\in C^\infty(\mathbb{R})$ be such that $\chi(t)=1$ for all $t\in(3,5)$ and $\operatorname{supp}(\chi)\subset[2,6]$. We  define $\eta\in C^\infty(\mathbb{R}^3)$ by
	$$
	\eta(x)=\sum_{k=0}^\infty\chi(10^{-k}\,|x|).
	$$
	Consider the metric
	$$
	g=\left(1+|x|^{-1}-\frac18\,\eta(x)\,x^3\,|x|^{-4}\right)^4\bar g
	$$
	 on $\mathbb{R}^3\setminus\{0\}$. Note that
	$$
	g=\bigg[(1+|x|^{-1})^{-4}+O(|x|^{-3})\bigg]\,\bar g.
	$$
	It follows that the limit in \eqref{center of mass} exists and that $C=0$. \\ \indent Let $k\geq 1$ be an integer and suppose that $x\in\mathbb{R}^3$ with $3<10^{-k}\,|x|<5$. Using that $\eta=1$ near $x$, we compute
	\begin{align} \label{R does not vanish} 
	R(x)=\sum_{i=1}^3\bar D^2_{e_i,e_i}(x^3\,|x|^{-4})=4\,x^3\,|x|^{-6}.
	\end{align} 
	Conversely, if $6<10^{-k}\,|x|<8$, there holds $\eta= 0$ near $x$. We find that
	\begin{align*} 
	R(x)=0.
	\end{align*}
	Let $\lambda_{k}=4\cdot10^k$ and $\hat \lambda_{k}=7\cdot10^k$. 
	Using Lemma \ref{xi estimate}, that $\bar DR=O(|x|^{-6})$, and \eqref{R does not vanish}, we compute
	$$
	\lambda_{k}^3\,\int_{S_{\lambda_{k},\xi(\lambda_{k})}}R\,\bar \nu\,\mathrm{d}\bar\mu=\lambda_{k}^3\,\int_{S_{\lambda_{k}}(0)}R\,\bar \nu\,\mathrm{d}\bar\mu+O(10^{-k})
	=
	\frac{16\,\pi}{3}\,e_3+O(10^{-k}).
	$$
	In conjunction with $C=0$ and Corollary \ref{com coro}, we find
	$$
	\lambda_k\,\xi(\lambda_k)=\frac{1}{24}\,e_3+O(10^{-k}).
	$$
	Likewise, we obtain
	$$
	\hat\lambda_k\,\xi(\hat\lambda_k)=O(10^{-k});
	$$
	see Figure \ref{comccounterfigure}.
	It follows from this and \eqref{com relation} that the limit in \eqref{acwcom def} does not exist.
\end{proof}

\section{Proof of Theorem \ref{general existence thm}}
Throughout this section, we assume that $(M,g)$ is $C^4$-asymptotic to Schwarzschild with mass $m=2$. \\
\indent Let $\delta\in(0,1/2)$. We recall the definitions \eqref{G definition}, \eqref{G1 definition}, and \eqref{G2 definition} of the functions $$G_\lambda,\, G_1,\, G_{2,\lambda}:\{\xi\in\mathbb{R}^3:|\xi|<1-\delta\}\to\mathbb{R}.$$ Moreover, recall from \eqref{G decomposition} that  
	\begin{align*} 
G_\lambda(\xi)=G_1(\xi)+G_{2,\lambda}(\xi)+O(\lambda^{-1}).
\end{align*}
\indent According to Lemma \ref{g1 properties}, the function $G_1$ is strictly convex. We now identify a useful convexity criterion for functions that resemble $G_{2,\lambda}$.    
\begin{lem}
	Let $f\in C^{1}(\mathbb{R}^3)$ be a non-negative function satisfying
	\begin{align}
\sum_{i=1}^3	x^i\,\partial_i(|x|^2\,f)\leq0. \label{exact growth}
	\end{align} 
	For every $\xi_1,\xi_2\in\mathbb{R}^3$ with $|\xi_1|,\,|\xi_2|<1$ and $\lambda>0$ there holds
 \label{monotone gradient lemma}
$$
\int_{S_{\xi_1,\lambda}}\bar g(\bar\nu,\xi_2-\xi_1)\,f\,\mathrm{d}\bar\mu\geq\int_{S_{\xi_2,\lambda}}\bar g(\bar\nu,\xi_2-\xi_1)\,f\,\mathrm{d}\bar\mu.
$$

\end{lem}
\begin{proof} By scaling, we may assume that $\lambda=1$. Moreover, we may assume that $\xi_2\neq\xi_1$. We define the hemispheres 
	$$
	S_+^\ell=\{x\in S_{1}(\xi_\ell):\bar g(\bar\nu,\xi_2-\xi_1)\geq 0\} \qquad \text{and} \qquad 	S_-^\ell=\{x\in S_{1}(\xi_\ell):\bar g(\bar\nu,\xi_2-\xi_1)\leq 0\}
	$$
	where $\ell=1,2$. We choose an orthonormal basis $\{e_1,\,e_2,\,e_3\}$ of $\mathbb{R}^3$ with $e_1\perp \operatorname{span}\{\xi_1,\xi_2\}$ and
	$$
	e_3=\frac{\xi_2-\xi_1}{|\xi_2-\xi_1|}
	$$ 
	and parametrize almost all of $S_2^+$ via
	$$
\Psi:(0,\pi)\times(0,2\,\pi)\to S_2^+ \qquad\text{given by}\qquad 	\Psi(\zeta,\varphi)= \xi_2+(\sin\zeta\,\sin\varphi,\sin\zeta\,\cos\varphi,\cos\zeta).
	$$
	Likewise, we parametrize almost all of $S_1^+$ by 
	$$
	(0,\pi)\times(0,2\,\pi)\to S_1^+\qquad\text{where}\qquad	(\theta,\varphi)\mapsto \xi_1+(\sin\theta\,\sin\varphi,\sin\theta\,\cos\varphi,\cos\theta).
	$$
	\indent It is geometrically evident and straightforward to check  that, given $\zeta$, there is $\theta=\theta(\zeta)$ with $\theta\leq\zeta$ and $t=t(\zeta)>1$ such that
	\begin{align}\label{multipliable equation}
	t\,\left[\xi_1+(\sin\theta\,\sin\varphi,\sin\theta\,\cos\varphi,\cos\theta)\right]=\xi_2+(\sin\zeta\,\sin\varphi,\sin\zeta,\cos\varphi,\cos\zeta); 
	\end{align}
	see Figure \ref{convexity}.
	\begin{figure}\centering
		\includegraphics[width=0.4\linewidth]{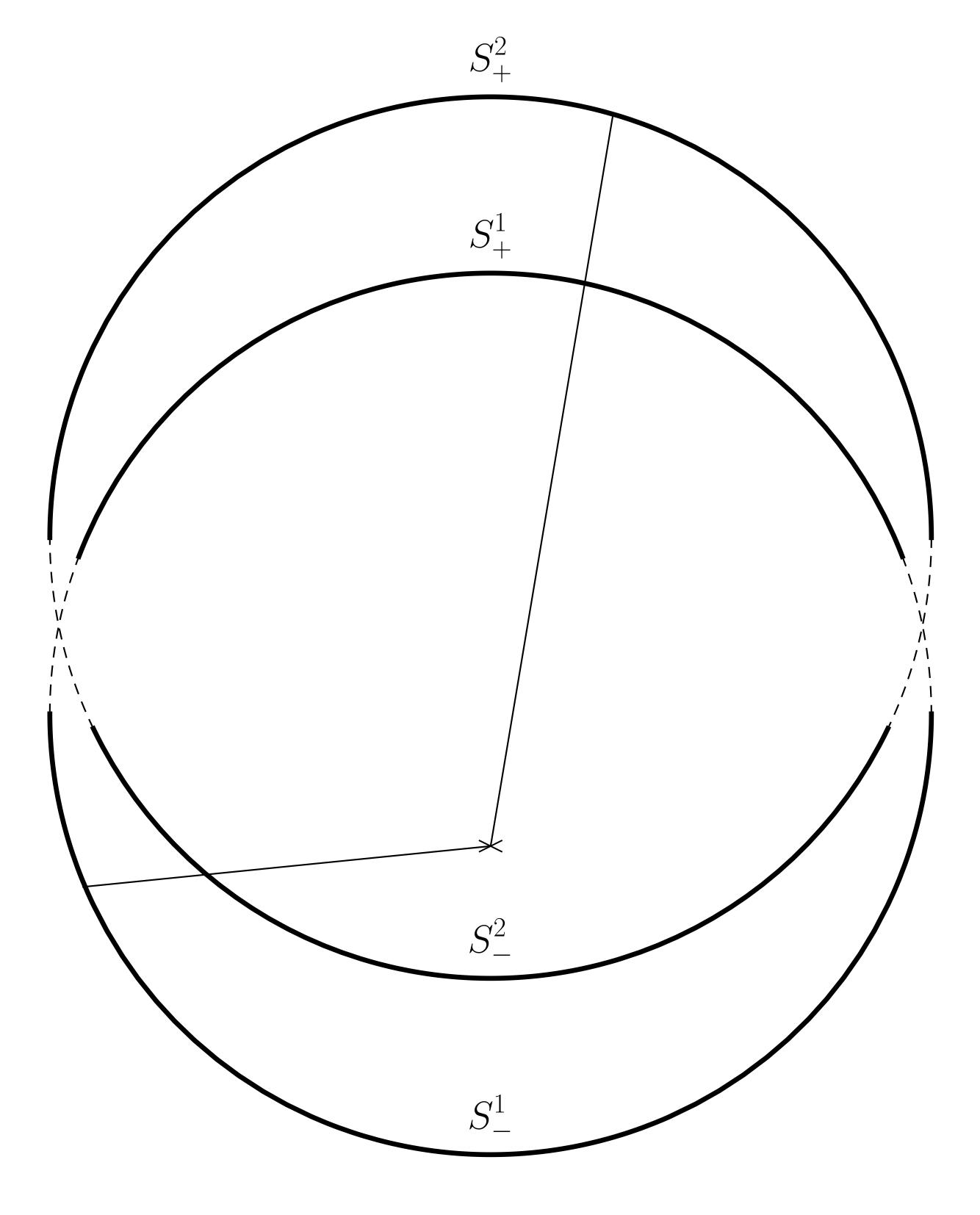}
		\caption{An illustration of the proof of Lemma \ref{monotone gradient lemma}. The function $f$ is compared along the lines connecting $S^1_\pm$ and $S^2_\pm$. The cross marks the origin of $\mathbb{R}^3$.  }
		\label{convexity}
	\end{figure}
		We define
	$
	a=\bar g(\xi_1,e_3)$ and $b=\bar g(\xi_2,e_3).
	$
	Dotting \eqref{multipliable equation} with $e_1$, we obtain
	$$
	t=\frac{\sin\zeta}{\sin\theta}.
	$$
Likewise, dotting \eqref{multipliable equation} with $e_3$, we find that
	$$
	t=\frac{\cos\zeta+b}{\cos\theta+a}.
	$$
	In particular, we obtain the relation
	\begin{align} \label{differentiable equation}
	\frac{\cos\zeta+b}{\cos\theta+a}=\frac{\sin\zeta}{\sin\theta}.
	\end{align}
	Differentiating \eqref{differentiable equation} with respect to $\zeta$, we find
	$$
	-\frac{\sin\zeta}{\cos\theta+a}+t\,\frac{\sin\theta}{\cos\theta+a}\,\dot\theta=\frac{\cos\zeta}{\sin\theta}-t\,\frac{\cos\theta}{\sin\theta}\,\dot\theta.
	$$
	Equivalently,
	$$
	\dot\theta = t^{-1}\,\frac{\cos\theta\,\cos\zeta+\cos\zeta\,a+\sin\zeta\,\sin\theta}{1+a\,\cos\theta}.
	$$
Using $\zeta\geq\theta$, we obtain that
	$$
	\cos\theta\,\frac{\cos\theta\,\cos\zeta+\cos\zeta\,a+\sin\zeta\,\sin\theta}{1+a\,\cos\theta}\geq \cos\zeta.
	$$

It follows that
	$$
	\dot\theta\,\sin\theta\,\cos\theta\geq t^{-2}\, \sin\zeta\,\cos\zeta.
	$$
\indent 	Using that $f$ is non-negative and \eqref{exact growth}, it follows that
	\begin{equation}  \label{c1}
	\begin{aligned}
	&\int_{S^1_{+}}f\,\bar g(\bar\nu,\xi_2-\xi_1)\,\mathrm{d}\bar\mu -\int_{S^2_{+}}f\,\bar g(\bar\nu,\xi_2-\xi_1)\,\mathrm{d}\bar\mu
\\	& \qquad\geq |\xi_2-\xi_1|\int_0^{2\,\pi}\int_0^\pi \left[t^{-2}\,f(t^{-1}\,\Psi(\zeta,\varphi))-f(\Psi(\zeta,\varphi))\right]\,\sin\zeta\,\cos\zeta\,\mathrm{d}\zeta\,\mathrm{d}\varphi
\\& \qquad \geq 0.
	\end{aligned}
	\end{equation}
	The same argument shows that
	\begin{equation} \label{c2}
	\int_{S^1_{-}}f\,\bar g(\bar\nu,\xi_2-\xi_1)\,\mathrm{d}\bar\mu -\int_{S^2_{-}}f\,\bar g(\bar\nu,\xi_2-\xi_1)\,\mathrm{d}\bar\mu\geq 0.
	\end{equation}
\indent 	The assertion of the lemma follows from \eqref{c1} and \eqref{c2}.
\end{proof}
\begin{coro}
	Let $\delta\in(0,1/2)$ and suppose that $f\in C^{1}(\mathbb{R}^3)$ satisfies, as $x\to\infty$,
$$
f\geq-o(|x|^{-4})\qquad\text{and}\qquad \sum_{i=1}^3x^i\,\partial_i(|x|^2\,f)\leq o(|x|^{-2}).
$$
There holds,  \label{monotone gradient}  uniformly  for all $\xi_1,\,\xi_2\in\mathbb{R}^3$ with $|\xi_1|,\,|\xi_2|<1-\delta$ as $\lambda\to\infty$,
$$
\int_{S_{\xi_1,\lambda}}f\,\bar g(\bar\nu,\xi_2-\xi_1)\,\mathrm{d}\bar\mu\geq \int_{S_{\xi_2,\lambda}}f\,\bar g(\bar\nu,\xi_2-\xi_1)\,\mathrm{d}\bar\mu-o(\lambda^{-2}).
$$

\end{coro} 
\begin{proof}
	This follows from Lemma \ref{monotone gradient lemma} applied to the function 
	$
	f_\epsilon=f+\epsilon\,|x|^{-4}
	$ for appropriate choice of $\epsilon>0$.

\end{proof}
\begin{proof}[Proof of Theorem \ref{general existence thm}]
	First, suppose that $R\geq- o(|x|^{-4})$. The argument presented in the proof of \cite[Theorem 5]{acws} shows that there  
	 exists a family $\{\Sigma(\kappa):\kappa\in(0,\kappa_0)\}$ of  area-constrained Willmore spheres $\Sigma(\kappa)\subset M$ such that \eqref{constrained Willmore quantity} holds with parameter $\kappa$ and such that \eqref{required properties} holds.
	 \\ \indent To prove the uniqueness statement, suppose that
	 $$
	 \sum_{i=1}^3x^i\,\partial_i(|x|^2\,R)\leq o(|x|^{-2})
	 $$
	 and let $\delta\in(0,1/2)$. It follows from \eqref{G decomposition}, Lemma \ref{g1 properties}, and Corollary \ref{monotone gradient} that $G_\lambda$ is strictly convex provided $\lambda>1$ is sufficiently large. In particular, $G_\lambda$ has at most one critical point. We can now argue exactly as in the proof of \cite[Theorem 8]{acws}.
\end{proof}
\begin{rema}
	Suppose that $$\sum_{i=1}^3x^i\,\partial_i(|x|^2\,R)\leq o(|x|^{-2})$$ and let \label{positioning remark} $\xi(\lambda)\in\mathbb{R}^3$ be the unique critical point of $G_\lambda$  constructed in the proof of Theorem \ref{general existence thm}. Using Lemma \ref{G lemma}, we find that
	$$
\xi(\lambda)=2\,\lambda^2\,	|(\bar DG_1)({\xi(\lambda)})|^{-1}\,\int_{S_{\xi(\lambda),\lambda}}R\,\bar\nu\,\mathrm{d}\bar\mu+O(\lambda^{-1}).
	$$ 
	In particular, up to lower-order terms, the positioning of the asymptotic family by area-constrained Willmore surfaces \eqref{acw foliation} is determined by the asymptotic distribution of  scalar curvature.
\end{rema}
\section{Proof of Theorem \ref{existence counter} and Theorem \ref{slow divergence counter}} 
We recall the definitions \eqref{G1 definition} of $G_1$ and \eqref{G2 definition} of  $G_{2,\lambda}$. A direct computation shows that
\begin{align} \label{dg1 counter intro}
(\bar DG_1)({\xi})=2\,\pi\,\left[8\,(1-|\xi|)^{-2}+40\,(1-|\xi|)^{-1}-24\,\log(1-|\xi|)\right]\xi+O(1)		\end{align} 
as $|\xi|\nearrow1$. \\ \indent 
To prove Theorem \ref{existence counter} and Theorem \ref{slow divergence counter}, we construct suitable metrics $g$  on $\mathbb{R}^3\setminus\{0\}$ such that the Schwarzschild contribution \eqref{dg1 counter intro} cancels with that from $G_{2,\lambda}$ for suitable $\lambda>1$. We then adjust $g$ accordingly to force the non-existence respectively existence of  large area-constrained Willmore spheres. \\ \indent 
First, we choose a function $\chi:\mathbb{R}\to[0,1]$ with $\operatorname{supp}(\chi)\subset(1/2,4)$  and $\chi(t)=1$ if $t\in[3/4,3]$. \\ \indent Let $k$ and $\ell$ be non-negative integers. We define $\chi_k:\mathbb{R}\to[0,1]$ by
\begin{equation} \label{chi k def intro} 
\chi_k(t)=\begin{dcases} &\chi(t) \qquad\quad\,\,\,\,\text{if } t\leq 1,\\
&1 \quad\,\qquad\quad\,\,\,\, \text{if }1<t<k^2, \\ 
&\chi(k^{-2}\,t) \qquad\text{if } t>k^2.
\end{dcases}
\end{equation}
Note that
\begin{align} \label{chi support intro}
\operatorname{supp}(\chi_k)\subset[1/2,4\,k^2].
\end{align}
Let $
\lambda_{k,\ell}=k^2\,10^{\ell^2}. 
$ 
Given $a_1,\,a_2,\,a_3,\,a_4\in\mathbb{R}$, we define
\begin{equation*}
\begin{aligned}  
\eta_{k,\ell}=\,& \chi_k(10^{-\ell^2}\,|x|)\bigg[a_1\,|x|^{-2}+a_2\,\lambda_{k,\ell}^{-1}\,|x|^{-1}\,\big(\log \lambda_{k,\ell}-\log|x|\big)+a_3\,\lambda_{k,\ell}^{-2}\,\big(\log|x|-\log \lambda_{k,\ell}\big)\\&\qquad\qquad\qquad\quad +a_4\,\lambda_{k,\ell}^{-5}\,(x^3)^3\bigg].
\end{aligned} 
\end{equation*} 
Note that
\begin{align} \label{log estimate intro} 
\lambda_{k,\ell}^{-1}\,|x|^{-1}\,|\log|x|-\log \lambda_{k,\ell}|< 100\,|x|^{-2}
\qquad \text{and} \qquad 
\lambda_{k,\ell}^{-2}\,|\log|x|-\log\lambda_{k,\ell}|< 100\,|x|^{-2}
\end{align} 
provided $1/2\,k^{-2}\leq\lambda_{k,\ell}^{-1}\,|x|\leq 4$. Using \eqref{chi k def intro} and \eqref{log estimate intro}, we find that, for every multi-index $J$, there are universal constants $c_J>1$ such that 
\begin{align}  \label{uniform eta estimate}
|\partial_J\eta_{k,\ell}|\leq c_J\,(|a_1|+|a_2|+|a_3|+|a_4|)\,|x|^{-2-|J|}.
\end{align} 
\indent Let $x\in\mathbb{R}^3$ with $k^{-2}\,\leq\lambda_{k,\ell}^{-1}\,|x|\leq 2$. By \eqref{chi support intro}, we have
\begin{align} 
\chi_k(10^{-j^2}\,|x|)=\delta_{\ell j} \label{chikij intro}
\end{align} 
for every $j$ provided $\ell$ is sufficiently large. Moreover, we compute
\begin{align} \label{laplacian}
\sum_{i=1}^3(\bar D^2_{e_i,e_i} \eta_{k,\ell})(x) =2\,a_1\,|x|^{-4}+a_2\,\lambda_{k,\ell}^{-1}\,|x|^{-3}+a_3\,\lambda_{k,\ell}^{-2}\,|x|^{-2}+6\,a_4\,\lambda_{k,\ell}^{-5}\,x^3.
\end{align} 
Fix $\xi\in\mathbb{R}^3$ with $|\xi|<1$.  We compute\begin{equation}
	\label{der g2l 1}
\begin{aligned} 
&\lambda_{k,\ell}^2\,	\int_{S_{\xi,\lambda_{k,\ell}}} \left[2\,a_1\,|x|^{-4}+a_2\,\lambda_{k,\ell}^{-1}\,|x|^{-3}+a_3\,\lambda_{k,\ell}^{-2}\,|x|^{-2}\right]\bar\nu\,\mathrm{d}\bar\mu
\\&\qquad =
\int_{S_{1}(\xi)} \left[2\,a_1\,|x|^{-4}+a_2\,|x|^{-3}+a_3\,\,|x|^{-2}\right]\bar\nu\,\mathrm{d}\bar\mu
\\&\qquad = -2\,\pi\,\left [a_1\,(1-|\xi|)^{-2}+(a_1+a_2)\,(1-|\xi|)^{-1}+(a_1-a_3)\,\log(1-|\xi|)\right]\xi\\&\qquad \qquad+\sum_{i=1}^3a_i\,f_i(\xi)\,\xi
\end{aligned}
\end{equation}
where $f_1,\,f_2,\,f_3\in C^\infty(B_1(0))$  are bounded. Likewise, 
\begin{align} \label{der g2l 2}
\lambda_{k,\ell}^2\,\int_{S_{\xi,\lambda_{k,\ell}}}6\,a_4\,\lambda_{k,\ell}^{-5}\,x^3\,\bar\nu\,\mathrm{d}\mu=8\,\pi\,a_4\,e_3.	\end{align} 
 \indent Now, suppose that
$$
g=\left(1+|x|^{-1}+\frac12\,\eta_{k,\ell}\right)^4\bar g.
$$
Note that $$
R=-4\,\sum_{i=1}^3\bar D^2_{e_i,e_i}\eta_{k,\ell} 
$$
and recall the definition \eqref{G2 definition} of $G_{2,\lambda_{k,\ell}}$. Assume that  $|\xi|<1-k^{-2}$. Using \eqref{chikij intro}, \eqref{laplacian},  \eqref{der g2l 1}, and \eqref{der g2l 2}, we conclude
\begin{equation} \label{dg2 counter intro}
\begin{aligned}
(\bar DG_{2,\lambda_{k,\ell}})({\xi})=\,&-16\,\pi\,\bigg [a_1\,(1-|\xi|)^{-2}+(a_1+a_2)\,(1-|\xi|)^{-1}+(a_1-a_3)\,\log(1-|\xi|)\bigg]\xi
\\&+64\,\pi\,a_4\,e_3
\\&+8\,\sum_{i=1}^3a_i\,f_i(\xi)\,\xi
\end{aligned} 
\end{equation} 
for every sufficiently large $\ell$.
We emphasize the structural similarity of \eqref{dg1 counter intro} and \eqref{dg2 counter intro}.
 \begin{proof}[Proof of Theorem \ref{existence counter}]
 Let
		$$
				g=\left(1+|x|^{-1}+\frac12\sum_{i=1}^\infty\eta_{i,i}\right)^4\,\bar g.
		$$
		Using \eqref{uniform eta estimate}, we find that $g$ is $C^k$-asymptotic to the Schwarzschild metric with mass $m=2$ for every $k\geq 2$.  \\ \indent  We choose $a_1=1$, $a_2=4$ and $a_3=4$. Let $\delta\in(0,1/2)$. Recalling \eqref{G decomposition} and using \eqref{dg2 counter intro} and \eqref{dg1 counter intro}, we obtain,  uniformly for every $\xi\in\mathbb{R}^3$ with $|\xi|<1-\delta$ as $i\to\infty$,
		\begin{align} \label{dg counter}
		(\bar D G_{\lambda_{i,i}})({\xi}) =64\,\pi\,a_4\,e_3+O(1).
		\end{align}
 \indent Suppose that there exists a family 
		$\{\Sigma(\kappa):\kappa\in(0,\kappa_0)\}$ of area-constrained Willmore spheres $\Sigma\subset\mathbb{R}^3\setminus\{0\}$ enclosing the origin and  satisfying \eqref{constrained Willmore quantity}  with parameter $\kappa$ such that 
		$$
		\lim_{\kappa\to0} \rho(\Sigma(\kappa))=\infty,\qquad 
		\limsup_{\kappa\to0} \rho(\Sigma(\kappa))^{-1}\,\lambda(\Sigma(\kappa))<\delta^{-1},\qquad\text{and}\qquad \lim_{\kappa\to0} \int_{\Sigma(\kappa)}|\hcirc|^2\,\mathrm{d}\mu=0.
		$$
		Arguing as in the proof of \cite[Theorem 8]{acws}, we find that the function $G_{\lambda_{i,i}}$ has a critical point $\xi_i$ with
		$|\xi_i|<1-\delta$ for every sufficiently large integer $i$.  This is  incompatible with \eqref{dg counter} if $a_4>1$ is chosen sufficiently large.
	\end{proof} 
For the proof of Theorem \ref{slow divergence counter}, we argue in two steps.
\begin{lem} \label{sd lemma}
There are   constants $c_J>1$ such that the following holds. 
	For every $\delta\in(0,1/2)$, there exists a metric $g$ on $\mathbb{R}^3\setminus\{0\}$ that is $C^k$-asymptotic to Schwarzschild with mass $m=2$ for every $k\geq 2$ with
	\begin{align} \label{uniform decay}
\limsup_{x\to\infty}|x|^{2+|J|}	\,|\partial_J \,\sigma|<c_J
	\end{align} 
	for every multi-index $J$
 that satisfies the following property. \\ \indent  There exists a sequence $\{\Sigma_i\}_{i=1}^\infty$ of area-constrained Willmore spheres $\Sigma_i\subset \mathbb{R}^3\setminus\{0\}$ such that
	$$
	\lim_{i\to\infty} \rho(\Sigma_i)=\infty
	$$
and 	$\lambda(\Sigma_i)^{-1}\,\Sigma_i$ converges smoothly to a round sphere
	while
	$$
	 \rho(\Sigma_i)<\delta\,\lambda(\Sigma_i) \qquad\text{and} \qquad  
	m_H(\Sigma_i)>2
	$$
	for all $i$.
\end{lem} 
\begin{proof}
Let $k$ be a positive integer
	and define the metric
$$
g=\left(1+|x|^{-1}+\frac12\sum_{i=1}^\infty\eta_{k,i}\right)^4\bar g.
$$
Note that \eqref{uniform eta estimate} implies \eqref{uniform decay}. \\ \indent 
We choose $a_1=2$, $a_2=3$, $a_3=5$, and $a_4=0$. Using \eqref{dg1 counter intro} and \eqref{dg2 counter intro}, we find that,  uniformly for every $\xi\in\mathbb{R}^3$ with $|\xi|<1-k^{-2}$ as $k\to\infty$, 
	\begin{align}  \label{derivative sd}
\sum_{j=1}^3\xi^j\,(\partial_j [G_{1}+G_{2,\lambda_{k,i}}])({\xi}) =-16\,\pi\,(1-|\xi|)^{-2}+O(1)
\end{align} 
  provided $i$ is sufficiently large. Recalling \eqref{G decomposition}, we conclude that for every large $k$  there holds
\begin{align} \label{sd 1}
\sum_{j=1}^3\xi^j\,(\partial_j G_{\lambda_{k,i}})({\xi}) <0
\end{align}
for every  sufficiently large $i$ and every $\xi\in\mathbb{R}^3$ with $|\xi|=1-2\,k^{-2}$. 
\\ \indent By contrast,  it follows from \eqref{chi support intro} that $R(x)=0$  if $10^{-2\,i}<\lambda_{k,i}^{-1}\,|x|<1/2\,k^{-2}$. In conjunction with the estimate $R=O(|x|^{-4})$, we conclude from \eqref{G2 definition} that, as $i\to\infty$ for every $\xi\in\mathbb{R}^3$ with $1-1/2\,k^{-2}<|\xi|<1-10^{-2\,i}$,
$$
(\bar D G_{2,\lambda_{k,i}})({\xi})=O(k^8).
$$ 
 Recalling \eqref{dg1 counter intro} and using \eqref{G decomposition}, we conclude that there is $\delta(k)\in(0,k^{-2})$ such that
\begin{align} \label{sd 2}
\sum_{j=1}^3\xi^j\,(\partial_j G_{\lambda_{k,i}})(\xi)>0
\end{align} 
for every $\xi\in\mathbb{R}^3$ with $|\xi|=1-\delta(k)$ and every  sufficiently large integer $i$. Together with the fact that $g$ is rotationally symmetric, \eqref{sd 1} and \eqref{sd 2} imply that for every $i$ sufficiently large, $G_{\lambda_{k,i}}$ has a local minimum $\xi_i\in\mathbb{R}^3$ with $1-2\,k^{-2}<|\xi_i|<1$. \\
\indent Finally, we observe that $G_{\lambda_{k,i}}(0)=O(1)$. Using   \eqref{derivative sd} and \eqref{G decomposition}, we conclude that $G_{\lambda_{k,i}}(\xi_i)<0$ for every sufficiently large $i$.   The assertions of the lemma follow from Proposition \ref{LS prop}, \eqref{G definition}, and the definition of the Hawking mass \eqref{hawking mass}.
\end{proof} 

\begin{proof}[Proof of Theorem \ref{slow divergence counter}]
	Using Lemma \ref{sd lemma}, we may choose a sequence $\{g_{k}\}_{k=1}^\infty$ of Riemannian metrics $g_k$  on $\mathbb{R}^3\setminus\{ 0\}$ that are $C^k$-asymptotic to Schwarzschild with mass $m=2$ for every $k\geq 2$ and satisfy \eqref{uniform decay} such that the following holds. There is a sequence $\{\Sigma_k\}_{k=1}^\infty$ of spheres $\Sigma_k\subset \mathbb{R}^3\setminus \{0\}$ with the following four properties.
	\begin{itemize}
		\item[$\circ$] For every positive integer $k$, there holds 	\begin{align} \label{sd ratio} 
		\rho(\Sigma_{k+1})>10\,\Theta(\Sigma_k)
		\end{align} 
		where $\Theta(\Sigma_k)=\sup\{|x|:x\in \Sigma_k\}$
		is the outer radius of $\Sigma_k$.
		\item[$\circ$] 	$\Sigma_k$  is an area-constrained Willmore surface with Hawking mass $m_H(\Sigma_k)>2$ with respect to $g_k$.
\item[$\circ$] 		$\lambda(\Sigma_k)^{-1}\,\Sigma_k$ converges smoothly to a round sphere.
\item[$\circ$] There holds $\rho(\Sigma_k)<k^{-1}\,	\lambda(\Sigma_k)$ for every positive integer $k$.
	\end{itemize}
 \indent  Now, we choose a smooth function $\gamma:\mathbb{R}\to[0,1]$ with $\operatorname{supp}(\gamma)\subset[1/3,3]$ and $\gamma(t)=1$ for $t\in[1/2,2]$ and define $\gamma_k:\mathbb{R}\to[0,1]$
	\begin{align} \label{sd chi}
	\gamma_k(t)=\begin{dcases}&\gamma(\rho(\Sigma_k)^{-1}\,t) \,\qquad\text{if } t<\rho(\Sigma_k) \\
	&1\,\,\,\,\,\qquad\qquad\,\qquad \text{if } \rho(\Sigma_k)\leq t \leq \Theta(\Sigma_k)\\ &
	\gamma(\Theta(\Sigma_k)^{-1}\,t)\qquad \text{if } t>\Theta(\Sigma_k).
		\end{dcases}
	\end{align} 
	By \eqref{sd ratio}, there holds $\operatorname{supp}(\gamma_k)\cap\operatorname{supp}(\gamma_j)=\emptyset$ whenever $k\neq j$. Consider the Riemannian metric 
	$$
	g=(1+|x|^{-1})^4\,\bar g+\sum_{k=0}^\infty \gamma_k(|x|)\,(g_k-(1+|x|^{-1})^4\,\bar g)
	$$
	on $\mathbb{R}^3\setminus\{0\}$. 	Using \eqref{uniform decay} and \eqref{sd chi}, we find that $g$ is $C^k$-asymptotic to Schwarzschild with mass $m=2$ for every $k\geq 2$. Moreover, there holds $g=g_k$ near $\Sigma_k$. The assertions follow.
\end{proof} 

\begin{appendices}
	\section{The Hamiltonian center of mass}
	In this section, we recall some facts on the Hamiltonian center of mass of an asymptotically flat manifold $(M,g)$. \\ \indent
	Let $(M,g)$ be a complete, non-compact Riemannian $3$-manifold with integrable  scalar curvature $R$. Given an integer $k\geq 2$ and  $\tau>1/2$, we say that $(M,g)$ is $C^k$-asymptotically flat of rate $\tau$  if there is a non-empty compact set whose complement in $M$ is diffeomorphic to $\{x\in\mathbb{R}^3:|x|>1\}$ with,  for every multi-index $J$ with $|J|\leq k$ and as $x\to\infty$,
	$$
	g=\bar g+\sigma\qquad\text{where}\qquad \partial_J\sigma=O(|x|^{-\tau-|J|})
	$$
	  in this asymptotically flat chart. We usually fix such an asymptotically flat chart and use it as reference for statements on the decay of quantities. \\ \indent The  mass $m$ and the Hamiltonian center of mass $C=(C^1,\,C^2,\,C^3)$ of $(M,g)$ are given by \eqref{ADM mass} and \eqref{center of mass}, respectively. The limit in \eqref{ADM mass} is  well-defined for every such manifold $(M,g)$. The limits in \eqref{center of mass} exist if the metric $g$ satisfies additional asymptotic symmetry conditions.
	\begin{thm}[{\cite[Theorem 2.2]{Huang2}}] Suppose that $(M,g)$ is $C^2$-asymptotically flat of rate $\tau>1/2$ with,  for every multi-index $J$ with $|J|\leq2$ and as $x\to\infty$, \label{com existence criterion}
		\begin{align}
	\partial_J\left[g(x)-g(-x)\right]&=O(|x|^{-1-\tau-|J|}), \label{asymptotic symmetry of metric} \\
				R(x)-R(-x)&=O(|x|^{-7/2-\tau}). \notag
		\end{align}
		  Then the Hamiltonian center of mass \eqref{center of mass} of $(M,g)$ is well-defined.
	\end{thm} 
\begin{rema} \label{com existence} If $(M,g)$ is $C^2$-asymptotic to Schwarzschild, then \eqref{asymptotic symmetry of metric} holds for every $\tau\in(1/2,1]$.
\end{rema}

	\section{The geometric center of mass by large stable constant mean curvature spheres}
The study of existence  of large stable constant mean curvature spheres in asymptotically flat manifolds has been pioneered by G.~Huisken and S.-T.~Yau \cite{HuiskenYauCM}. There is a large body of important subsequent work; see e.g.~\cite{Ye,Metzger,Huang} and the references therein. The following general existence result has been established by C.~Nerz \cite{Nerz}. 
	\begin{thm}[{\cite[Theorems 5.1 and 5.2]{Nerz}}] Suppose that  $(M,g)$ is $C^2$-asymptotically flat of rate $\tau>1/2$ with positive mass $m>0$.	Then there exists a number $H_0>0$ and a foliation of the complement of a compact subset of $M$
		\begin{align} \label{cmc foliation}
		\{\Sigma_{CMC}(H):H\in(0,H_0)\}
		\end{align}
	such that $\Sigma_{CMC}(H)$ is a stable constant mean curvature sphere with mean curvature  $H$ for every $H\in(0,H_0)$. \label{cmc existence} 
	\end{thm} 
In \cite[\S 4]{HuiskenYauCM},   G.~Huisken and S.-T.~Yau have proposed to associate a geometric center of mass $$C_{CMC}=(C_{CMC}^1,C_{CMC}^2,C_{CMC}^3)$$ to  the foliation \eqref{cmc foliation}    where
	\begin{align} \label{CMCcom def} 
	C_{CMC}^\ell=\lim_{H\searrow0} |\Sigma_{CMC}(H)|^{-1}\int_{\Sigma_{CMC}(H)} x^\ell\,\mathrm{d}\mu,
	\end{align}
	provided the limit on the right-hand side exists for $\ell=1,\,2,\,3$. The existence of these limits has been studied for instance by J.~Metzger \cite{Metzger}, J.~Corvino and H.~Wu \cite{CorvinoWu}, L.-H.~Huang \cite{Huang2}, or J.~Metzger and the first-named author \cite{EichmairMetzgerIso}. In \cite{Nerz}, C.~Nerz has proven that the geometric center of mass \eqref{CMCcom def} coincides with the Hamiltonian center of mass \eqref{center of mass} of $(M,g)$, provided $g$ satisfies an additional asymptotic symmetry assumption.
	\begin{thm}[{\cite[Theorem 6.3]{Nerz}}]

	Suppose that, for every multi-index $J$ with $|J|\leq 2$ and as $x\to\infty$ ,  \label{nerzthm}
		\begin{align*}
		\partial_J\left[g(x)-g(-x)\right]&=O(|x|^{-1/2-\tau-|J|}),  \\ R(x)-R(-x)&=O(|x|^{-3- \tau}).
		\end{align*} 
		Then the limits in \eqref{center of mass} exist if and only if the limits in \eqref{CMCcom def} exist in which case $C=C_{CMC}$.
	\end{thm}\label{cmc appendix}
\begin{rema}
If $(M,g)$ is $C^2$-asymptotic to Schwarzschild, then the assumptions of \label{Schwarzschild gcom rema}  Theorem \ref{nerzthm} are satisfied; see \cite[Theorem 2]{Huang2}. 
\end{rema}

\begin{rema}  \label{cmc positioning}
	If $(M,g)$ is $C^5$-asymptotic to Schwarzschild with non-negative scalar curvature $R$ satisfying
	$$
	\sum_{i=1}^3 x^i\,\partial_i(|x|^2\,R)\leq 0,
$$ 
	then the leaves of the foliation \eqref{cmc foliation} are the only closed stable constant mean curvature surfaces $\Sigma\subset M$ with  large enclosed volume; see  \cite{acws} and also \cite{CCE,chodoshfar,ChodoshEichmairGlobal} for earlier work in this direction. According to Theorem \ref{nerzthm} and Remark \ref{Schwarzschild gcom rema}, the positioning of large stable constant mean curvature spheres is therefore governed by the Hamiltonian center of mass \eqref{center of mass} of $(M,g)$.
	\end{rema}  
\begin{rema}
	In \cite{cmc}, the authors give short alternative proofs for Theorem \ref{cmc existence} and Theorem \ref{nerzthm} based on Lyapunov-Schmidt reduction. The analysis carried out there is much less technical  than the one  in Section \ref{com section}. 
\end{rema}
\section{Lyapunov-Schmidt reduction} \label{LS appendix}
We review the construction of the foliation by large area-constrained Willmore spheres \eqref{acw foliation}  in \cite{acws}. Throughout, we will assume that $(M,g)$ is $C^4$-asymptotic to Schwarzschild with scalar curvature $R$; see \eqref{schwarzschild sigma decay}. 
\\ \indent Let $\delta\in(0,1/2)$. Given $\lambda>1$ and $\xi\in\mathbb{R}^3$, we consider the spheres \begin{align} \label{spheres def}S_{\xi,\lambda}=\{x\in\mathbb{R}^3:|x-\lambda\,\xi|=\lambda\}.\end{align} Given a function $u\in \Sigma_{\xi,\lambda}$, let \begin{align} \Sigma_{\xi,\lambda}(u)=\{x+\lambda^{-1}\,u(x)\,(x-\lambda\,\xi):x\in S_{\xi,\lambda}\}\label{radial graph}\end{align} be the Euclidean graph of $u$ over $S_{\xi,\lambda}$. Moreover, let $\Lambda_1(S_{\xi,\lambda})$ be the space of first spherical harmonics of $S_{\xi,\lambda}$ and $\Lambda_1^\perp(S_{\xi,\lambda})\subset C^\infty(S_{\xi,\lambda})$ be its orthogonal complement with respect to $L^2(S_{\xi,\lambda})$. 
\\ \indent
When stating that an error term  $$\mathcal{E}=O(\lambda^{-l_1}\,|\xi|^{l_3})+O(\lambda^{-l_2})$$ may be differentiated with respect to $\xi$  where $l_1,\,l_2>0$ and $l_3>1$, we mean that
$$
\bar{D} \mathcal{E}= O(\lambda^{-l_1}\,|\xi|^{l_3-1})+O(\lambda^{-l_2}),
$$
where differentiation is with respect to $\xi$. When stating that $\mathcal{E}$ may be differentiated with respect to $\lambda$, we mean that
$$
\mathcal{E}'= O(\lambda^{-l_1-1}\,|\xi|^{l_3})+O(\lambda^{-l_2-1})
$$\indent
 Let $\Sigma\subset M$ be a closed surface with $\lambda(\Sigma)=\lambda$, i.e.~$|\Sigma|=4\,\pi\,\lambda^2$. The normalized Willmore energy of $\Sigma$ is given by
\begin{align*}
F_\lambda(\Sigma)=\lambda^2\left(\int_{\Sigma}H^2\,\mathrm{d}\mu-16\,\pi-64\,\pi\,\lambda^{-1}\right).
\end{align*} Note that area-constrained Willmore surfaces $\Sigma\subset M$ are area-constrained critical points of $F_\lambda$. Moreover, recall from e.g.~\cite[Appendix A]{acws} that such surfaces are either minimal or satisfy the constrained Willmore equation
\begin{align*}
\kappa(\Sigma)\,W(\Sigma)=H(\Sigma)
\end{align*}
where 
\begin{align*}
 -W(\Sigma)=\Delta H+(|\hcirc|^2+\operatorname{Ric}(\nu,\nu))\,H
\end{align*}
and
\begin{align} \label{kappa function}
\kappa(\Sigma)=\left(\int_\Sigma H^2\,\mathrm{d}\mu\right)^{-1}\int_{\Sigma} \left[|\nabla H|^2-H^2\,|\hcirc|^2-H^2\,\operatorname{Ric}(\nu,\nu)\right]\mathrm{d}\mu.
\end{align}
\indent 
The Legendre polynomials $P_0,\,P_1,\,P_2,\dots$ are defined via a generating function.
Given $s\in[0,1]$ and $t\in[0,1)$, there holds
\begin{align*}
(1-2\,s\,t+t^2)^{-\frac12}=\sum_{\ell=0}^\infty P_l(s)\,t^\ell. 
\end{align*} 
The next proposition follows from \cite[Proposition 17]{acws}, \cite[Lemma 20]{acws}, and \cite[Lemma 21]{acws}.
\begin{prop}
	There are constants $\lambda_0>1$, $c>1$, and $\epsilon>0$ depending only on $(M,g)$ and $\delta\in(0,1/2)$ such that for every $\xi\in\mathbb{R}^3$ with $|\xi|<1-\delta$  and every $\lambda>\lambda_0$, there exists $u_{\xi,\lambda}\in \Lambda_1^\perp(S_{\xi,\lambda})$ and $\kappa_{\xi,\lambda}\in\mathbb{R}$  such that the following hold. The surface \begin{align} \label{sigma xi lambda def} \Sigma_{\xi,\lambda}=\Sigma_{\xi,\lambda}(u_{\xi,\lambda})
	\end{align} has area $|\Sigma_{\xi,\lambda}|=4\,\pi\, \lambda^2$. Moreover, $\Sigma_{\xi,\lambda}$ is an area-constrained Willmore surface with parameter $\kappa_{\xi,\lambda}$ if and only if $\xi$ is a critical point of the function  \label{LS prop}
	\begin{align}
	G_\lambda:\{\xi\in\mathbb{R}^3:|\xi|<1-\delta \}\to\mathbb{R}\qquad\text{given by} \qquad G_\lambda(\xi)=F_\lambda(\Sigma_{\xi,\lambda}). \label{G definition}
	\end{align}
\indent 	There holds,  uniformly for all $\xi\in\mathbb{R}^3$ with $|\xi|<1-\delta$ as $\lambda\to\infty$,
	\begin{align}
		u_{\xi,\lambda}&=-2+4\sum_{\ell=2}^\infty \frac{|\xi|^{\ell}}{\ell}\,P_l\left(-|\xi|^{-1}\,\bar{g}(y,\xi)\right)+O(\lambda^{-1}), \label{u estimate} \end{align}
			 and
		\begin{align}
	\kappa_{\xi,\lambda}&=4\,\lambda^{-3}+O(\lambda^{-4}). \label{kappa estimate}
	\end{align} 
The expansion	\eqref{u estimate} may be differentiated four times in spatial directions. The expansion \eqref{kappa estimate} may be differentiated once with respect to $\lambda$. \\ \indent 
	If  $\Sigma_{\xi,\lambda}(u)$ with $u\in \Lambda^\perp_1(S_{\xi,\lambda})$ is an area-constrained Willmore surface with  $|\Sigma_{\xi,\lambda}(u)|=4\,\pi\, \lambda^2$ and
	\begin{align*}
	|u|+\lambda\,|\nabla u|+\lambda^2\,|\nabla^2 u|+\lambda^3\,|\nabla^3 u|+\lambda^4\,|\nabla^4 u|&<\epsilon\,\lambda, \\\lambda^3\,|\kappa(\Sigma_{\xi,\lambda}(u))|&<\epsilon\,\lambda,
	\end{align*}
	then $u=u_{\xi,\lambda}$.
\end{prop}
The leading order term of the function $G_\lambda$ can be computed explicitly.
\begin{lem}[{\cite[Lemma 22]{acws}}] 
	There holds,  uniformly for every $\xi\in\mathbb{R}^3$ with $|\xi|<1-\delta $ as $\lambda\to\infty$, \label{G lemma}
	\begin{align} \label{G decomposition}
	G_\lambda(\xi)=G_1(\xi)+G_{2,\lambda}(\xi)+O(\lambda^{-1})
	\end{align}
	where
	\begin{align} \label{G1 definition}
	G_1(\xi)=64\,\pi+\frac{32\,\pi}{1-|\xi|^2}-48\,\pi\,|\xi|^{-1}\log\frac{1+|\xi|}{1-|\xi|}-128\,\pi\log(1-|\xi|^2)
	\end{align}
	and
	\begin{align} \label{G2 definition}
	G_{2,\lambda}(\xi)=2\,\lambda\,\int_{\mathbb{R}^3\setminus B_\lambda(\lambda\,\xi)}R\,\mathrm{d}\bar v.
	\end{align}
The expansion \eqref{G decomposition} may be differentiated twice with respect to $\xi$ and $\lambda$.
\end{lem}
We record some properties of the function $G_\lambda$.
\begin{lem} \label{g1 properties}
	The function $G_1$ is strictly convex and strictly increasing in radial directions. Moreover,  $G_1(0)=0$. 
\end{lem}
\begin{lem}[{\cite[Lemma 24]{acws}}] \label{g der lemma}
Suppose that, as $x\to\infty$, 
\begin{align*}
\sum_{i=1}^3x^i\,\partial_i(|x|^2\,R)\leq o(|x|^{-2})\qquad\text{and}\qquad
R(x)-R(-x)=o(|x|^{-4}).
\end{align*}
 Then,  uniformly for every $\xi\in\mathbb{R}^3$ with $|\xi|<1-\delta $ as $\lambda\to\infty$,	$$|\xi|^{-1}\,\sum_{i=1}^3\xi^i\,(\partial_iG_{2,\lambda})({\xi})\geq -o(1).$$
If the stronger decay assumptions
\begin{align*}
\sum_{i=1}^3x^i\,\partial_i(|x|^2\,R)\leq O(|x|^{-3})\qquad\text{and}\qquad
R(x)-R(-x)=O(|x|^{-5})
\end{align*}
hold, then,  uniformly for every $\xi\in\mathbb{R}^3$ with $|\xi|<1-\delta $ as $\lambda\to\infty$, 
$$|\xi|^{-1}\,\sum_{i=1}^3\xi^i\,(\partial_iG_{2,\lambda})({\xi})\geq -O(\lambda^{-1}).$$

\end{lem}
\begin{proof}
	This follows exactly as in \cite[Lemma 24]{acws}.
\end{proof}
The following proposition is a consequence of Lemma \ref{G lemma}, Lemma \ref{g1 properties}, and Lemma \ref{g der lemma}.
\begin{prop}[{\cite[Theorems 5 and 8]{acws}}] Suppose that, as $x\to\infty$, 
	\begin{align*}
\sum_{i=1}^3	x^i\,\partial_i(|x|^2\,R)\leq o(|x|^{-2})\qquad\text{and}\qquad 
	R(x)-R(-x)=o(|x|^{-4}).
	\end{align*}  There is $\lambda_0>1$ such that for every $\lambda>\lambda_0$ the function $G_\lambda$ has a unique critical point $\xi(\lambda)$. Moreover, \label{existence prop} as $\lambda\to\infty$,
	\begin{align*} 
	\xi(\lambda)=o(1).
	\end{align*}  
\end{prop}
\begin{rema} \label{kappa remark}
	It follows from \eqref{kappa estimate} that the function $$\lambda\mapsto \kappa(\Sigma(\lambda))
	$$
	is decreasing on $(\lambda_0,\infty)$, provided $\lambda_0>1$ is sufficiently large.
\end{rema}
\section{Some geometric expansions}
We collect several geometric computations that are needed in this paper. Throughout this section, we assume that $(M,g)$ is $C^4$-asymptotic to Schwarzschild with mass $m=2$.  We use a bar to indicate that a geometric quantity has been computed with respect to the Euclidean background metric $\bar g$. Likewise, we use the subscript $S$ to indicate that the Schwarzschild metric
$$
g_S=(1+|x|^{-1})^4\,\bar g
$$
with mass $m=2$ has been used in the computation.
\\ \indent Let $\delta\in(0,1/2)$, $\xi\in\mathbb{R}^3$ with $|\xi|<1-\delta$, and $\lambda>\lambda_0$, where $\lambda_0$ is the constant from Proposition \ref{LS prop}. Recall that $S_{\xi,\lambda}=S_\lambda(\lambda\,\xi)$.
\\ \indent The estimates below depend on $\delta\in(0,1/2)$ and $\lambda_0>1$ but are otherwise independent of $\lambda$ and $\xi$. 
\begin{lem}
	There holds
	\begin{align}
	\operatorname{Ric}_S(e_i,e_j)=2\,|x|^{-3}\,(\delta_{ij}-3\,x^i\,x^j\,|x|^{-2})+O(|x|^{-4})\label{rics}
	\end{align}
	and, as $x\to\infty$,
	 \label{Ricci expansion}
	\begin{equation} \label{ricci expansion equation}
	\begin{aligned}
	&(\operatorname{Ric}-{\operatorname{Ric}_S})(e_i,e_j)\\&\qquad=\,\frac12 \,\sum_{k=1}^3\left[(\bar D^2_{e_k,e_i}\sigma)(e_k,e_j)+(\bar D^2_{e_k,e_j}\sigma)(e_k,e_i) -(\bar D^2_{e_k,e_k} \sigma)(e_i,e_j)-(\bar D^2_{e_i,e_j}\sigma)(e_k,e_k)\right]\\&\qquad\qquad +O(|x|^{-5}).
	\end{aligned}	\end{equation}
\end{lem}
\begin{proof}
	\eqref{rics} follows from a direct computation. To prove \eqref{ricci expansion equation},  we define the family of metrics $
	g_t=g_S+t\,\sigma,
	$
	where $t\in[0,1]$. Note that $g_0=g_S$ and $g_1=g$. The estimate follows upon linearization of the expression $$\operatorname{Ric}(e_i,e_j)=\sum_{k=1}^3\bigg[\partial_k \Gamma^k_{ij}-\partial_i\Gamma^k_{kj}+\sum_{\ell=1}^3\big[\Gamma^k_{k\ell}\,\Gamma^\ell_{ij}-\Gamma^{k}_{il}\,\Gamma^\ell_{kj}\big]\bigg]$$
	where $\Gamma^{k}_{ij}$ are the Christoffel symbols of $g$. 
\end{proof}

\begin{lem}[{\cite[Lemma 41]{acws}}] There holds
	$$
	\nu_S(S_{\xi,\lambda})=(1+|x|^{-1})^{-2}\,\bar\nu
	$$
	and, uniformly for every $\xi\in\mathbb{R}^3$ with $|\xi|<1-\delta$  as $\lambda\to\infty$,
	\begin{align} \label{nu vs nus}
	\nu-\nu_S=\,\frac12\,\sigma(\bar\nu,\bar\nu)\,\bar\nu-\sum_{i=1}^3\sigma(\bar\nu,e_i)\,e_i+O(\lambda^{-3}).
	\end{align}
\end{lem}
Given $\xi\in\mathbb{R}^3$ and $\lambda>1$, we define the vector field
\begin{align} \label{Z definition}
Z_{\xi,\lambda}=(1+|x|^{-1})^{-2}\,\lambda^{-1}\,(x-\lambda\,\xi).
\end{align}
Note that $Z_{\xi,\lambda}=\nu_S(S_{\xi,\lambda})$ on $S_{\xi,\lambda}$.
\\ \indent For the statement of the next lemma, recall 
the definition of the conformal Killing operator $\mathcal{D}$  given by
\begin{align}
(\mathcal{D}Z)(X,Y)=g(D_XZ,Y)+g(D_YZ,X)-\frac23\,\operatorname{div}(Z)\,g(X,Y) \label{killing definition}
\end{align}
for vector fields $Z,\,X,\,Y$.
\begin{lem}
	There holds
\begin{equation} \label{schwarzschild killing}
		\begin{aligned} 
	(\mathcal{D}_S Z_{\xi,\lambda})(e_i,e_j)=\,&4\,\lambda^{-1}\,|x|^{-3}\,x^i\,x^j-2\,|x|^{-3}\,\left(x^i\,\xi^j+\xi^i\,x^j\right)+\frac43\,\left(|x|^{-3}\,\bar g(x,\xi)-\lambda^{-1}\,|x|^{-1}\right)\,\delta_{ij}\\&+O(|x|^{-3})+O(\lambda^{-1}\,|x|^{-2})
	\end{aligned}
\end{equation}
and,   uniformly for every $\xi\in\mathbb{R}^3$ with $|\xi|<1-\delta$ as $\lambda\to\infty$ on $\mathbb{R}^3\setminus B_{\lambda}(\lambda\,\xi)$, 
\begin{equation} \label{Killing changes}\begin{aligned}
(\mathcal{D}Z_{\xi,\lambda}-\mathcal{D}_SZ_{\xi,\lambda})(e_i,e_j)=\,&\sum_{k=1}^3\bigg[(\lambda^{-1}\,x^k-\xi^k)\,(\partial_k\sigma)(e_i,e_j)-\frac13\,(\lambda^{-1}\,x^k-\xi^k)\,(\partial_k\bar{\operatorname{tr}}\,\sigma)\,\delta_{ij}\bigg]\\&\qquad +O(\lambda^{-1}\,|x|^{-3})+O(|x|^{-4}).
\end{aligned}
\end{equation} 
\end{lem}
\begin{proof}
	To prove \eqref{schwarzschild killing}, we expand, as $\lambda\to\infty$ on $\mathbb{R}^3\setminus B_{\lambda}(\lambda\,\xi)$,
	$$
	Z_{\xi,\lambda}=\lambda^{-1}\,x-\xi-2\,\lambda^{-1}\,|x|^{-1}\,x+2\,|x|^{-1}\,\xi+O(\lambda^{-1}\,|x|^{-1})+O(|x|^{-2}).
	$$
	 	The first two terms on the right-hand side are conformal Killing vector fields. For the third term, we compute
	\begin{align*}
	(\mathcal{D}_S(\lambda^{-1}\,|x|^{-1}\,x))(e_i,e_j)=\,&\lambda^{-1}\,(\bar{\mathcal{D}}(|x|^{-1}\,x))(e_i,e_j)+O(\lambda^{-1}\,|x|^{-2})
	\\=\,&\frac23\,\lambda^{-1}\,|x|^{-1}\,\delta_{ij}-2\,\lambda^{-1}\,|x|^{-3}\,x^i\,x^j.
	\end{align*} Likewise, for the fourth term, we compute
	\begin{align*}
	(\mathcal{D}_S(|x|^{-1}\,\xi))(e_i,e_j)=\,&(\bar{\mathcal{D}}(|x|^{-1}\,\xi))(e_i,e_j)+O(|x|^{-3})\\=\,&-|x|^{-3}\,\big[\bar g(x,e_i)\,\bar g(\xi,e_j)+\bar g(\xi,e_i)\,\bar g(x,e_j)\big]+\frac23\,|x|^{-3}\,\bar g(x,\xi)\,\delta_{ij}+O(|x|^{-3}).
	\end{align*}\indent 
	To prove \eqref{Killing changes}, we again consider the family of metrics $
	g_t=g_S+t\,\sigma,
	$
	where $t\in[0,1]$, and linearize the expression
	$$
	(\mathcal{D}Z_{\xi,\lambda})(e_i,e_j)=g(D_{e_i}Z_{\xi,\lambda},e_j)+g(D_{e_j}Z_{\xi,\lambda},e_i)-\frac{2}{3}\,g(e_i,e_j)\,\operatorname{div}Z_{\xi,\lambda}.
	$$
\end{proof}

We recall the following expansion of the Willmore operator of a sphere $S_{\xi,\lambda}$.
\begin{lem}[{\cite[Corollary 45]{acws}}]
	There holds,  uniformly for every $\xi\in\mathbb{R}^3$ with $|\xi|<1-\delta $ as $\lambda\to\infty$, \label{willmore sphere}
	$$
	{W}({{S}_{\xi,\lambda}})=4\,\lambda^{-4}\sum_{\ell=0}^\infty (\ell-1)\,(\ell+1)\,(\ell+2)\,|\xi|^{\ell}\,P_l(-|\xi|^{-1}\,\bar g(\bar \nu,\xi))+{O}(\lambda ^{-5}).
	$$

\end{lem}
The following estimate is contained in the proof of \cite[Lemma 42]{acws}.
\begin{lem}
There holds,  uniformly for every $\xi\in\mathbb{R}^3$ with $|\xi|<1-\delta $ as $\lambda\to\infty$, \label{hcirc estimate}
	$$
	\int_{S_{\xi,\lambda}} |\hcirc|^2\,\mathrm{d}\mu=O(\lambda^{-4}).
	$$	This identity may be differentiated once with respect to $\xi$.
\end{lem}
The next lemma follows from Taylor's theorem and \cite[Lemma 31]{acws}. 
\begin{lem} \label{willmore expansion lemma}
	There exists a constant $\lambda_0>1$ such that for every $\lambda>\lambda_0$  the following holds. 	Let $u\in C^\infty(S_{\xi,\lambda})$ and suppose that there is $\epsilon>0$ with
	$$
| u|+\lambda\,|\nabla u|+\lambda^2\,|\nabla^2 u|+\lambda^3\,|\nabla^3 u|+\lambda^4\,	|\nabla^4 u|\leq \epsilon.
	$$ Then we have,  uniformly for every $\xi\in\mathbb{R}^3$ with $|\xi|<1-\delta$ as $\lambda\to\infty$,
	$$
	\int_{\Sigma_{\xi,\lambda}(u)}H^2\,\mathrm{d}\mu-\int_{S_{\xi,\lambda}}H^2\,\mathrm{d}\mu =-2\,\int_{S_{\xi,\lambda}}W\,u\,\mathrm{d}\mu +O(\epsilon^2\,\lambda^{-2}).
	$$
	This identity may be differentiated once with respect to $\xi$.
\end{lem}

\end{appendices}

\end{document}